\documentclass[12pt]{amsart}
\usepackage{amscd,amsmath,amsthm,amssymb}
\usepackage{pstcol,pst-plot,pst-3d}
\usepackage{lmodern,pst-node}
\usepackage{multicol}
\usepackage{epic,eepic}
\usepackage{amsfonts,amssymb,amscd,amsmath,enumerate,verbatim}
\psset{unit=0.7cm,linewidth=0.8pt,arrowsize=2.5pt 4}

\newpsstyle{fatline}{linewidth=1.5pt}
\newpsstyle{fyp}{fillstyle=solid,fillcolor=verylight}
\definecolor{verylight}{gray}{0.97}
\definecolor{light}{gray}{0.9}
\definecolor{medium}{gray}{0.85}
\definecolor{dark}{gray}{0.6}

\def\NZQ{\mathbb}               
\def\NN{{\NZQ N}}

\def\ZZ{{\NZQ Z}}
\def\RR{{\NZQ R}}

\def\BB{{\NZQ B}}

\def\frk{\mathfrak}               

\def\mm{{\frk m}}

\def\Phi{{\frk N}}
\def\Pc{{\mathcal P}}

\def\opn#1#2{\def#1{\operatorname{#2}}} 
\opn\chara{char} \opn\length{\ell} \opn\pd{pd} \opn\rk{rk}
\opn\projdim{proj\,dim} \opn\injdim{inj\,dim} \opn\rank{rank}
\opn\depth{depth} \opn\grade{grade} \opn\height{height}
\opn\size{size}
\opn\embdim{emb\,dim} \opn\codim{codim}

\opn\Tr{Tr} \opn\bigrank{big\,rank}
\opn\superheight{superheight}\opn\lcm{lcm}
\opn\trdeg{tr\,deg}
\opn\reg{reg} \opn\lreg{lreg} \opn\ini{in} \opn\lpd{lpd}
\opn\size{size}\opn{\mult}{mult}
\opn{\Cl}{Cl}
\opn\div{div} \opn\Div{Div} \opn\cl{cl} \opn\Cl{Cl}
\opn\Spec{Spec} \opn\Supp{Supp} \opn\supp{supp} \opn\Sing{Sing}
\opn\Ass{Ass} \opn\Min{Min} \opn\cl{cl}
\opn\Ann{Ann} \opn\Rad{Rad} \opn\Soc{Soc}
\opn\Syz{Syz} \opn\Im{Im} \opn\Ker{Ker} \opn\Coker{Coker}
\opn\Am{Am} \opn\Hom{Hom} \opn\Tor{Tor} \opn\Ext{Ext}
\opn\End{End} \opn\Aut{Aut} \opn\id{id} \opn\ini{in}

\opn\nat{nat}
\opn\pff{pf}
\opn\Pf{Pf} \opn\GL{GL} \opn\SL{SL} \opn\mod{mod} \opn\ord{ord}
\opn\Gin{Gin}
\opn\Hilb{Hilb}\opn\adeg{adeg}\opn\std{std}\opn\ip{infpt}
\opn\Pol{Pol}
\opn\sat{sat}
\opn\Var{Var}
\opn\Gen{Gen}
\opn\lex{lex}
\opn\div{div}

\opn\aff{aff} \opn\con{conv} \opn\relint{relint} \opn\st{st}
\opn\lk{lk} \opn\cn{cn} \opn\core{core} \opn\vol{vol}
\opn\link{link} \opn\star{star}
\opn\gr{gr}

\def\Bc{{\mathcal B}}

\def\Ic{{\mathcal I}}

\def\Fc{{\mathcal F}}
\def\Cc{{\mathcal C}}
\def\Ec{{\mathcal E}}
\def\Qc{{\mathcal Q}}
\def\Rc{{\mathcal R}}
\def\Mc{{\mathcal M}}
\def\pot#1#2{#1[\kern-0.28ex[#2]\kern-0.28ex]}

\opn\dirlim{\underrightarrow{\lim}}
\opn\inivlim{\underleftarrow{\lim}}
\let\sect=\cap

\let\iso=\cong

\let\to=\rightarrow

\def\Implies{\ifmmode\Longrightarrow \else
        \unskip${}\Longrightarrow{}$\ignorespaces\fi}
\def\implies{\ifmmode\Rightarrow \else
        \unskip${}\Rightarrow{}$\ignorespaces\fi}
\def\iff{\ifmmode\Longleftrightarrow \else
        \unskip${}\Longleftrightarrow{}$\ignorespaces\fi}

\let\:=\colon
\newtheorem{Theorem}{Theorem}[section]
\newtheorem{Lemma}[Theorem]{Lemma}
\newtheorem{Corollary}[Theorem]{Corollary}

\newtheorem{Remark}[Theorem]{Remark}

\newtheorem{Example}[Theorem]{Example}

\let\epsilon\varepsilon
\let\phi=\varphi
\let\kappa=\varkappa
\def\qed{\ifhmode\textqed\fi
      \ifmmode\ifinner\quad\qedsymbol\else\dispqed\fi\fi}
\def\textqed{\unskip\nobreak\penalty50
       \hskip2em\hbox{}\nobreak\hfil\qedsymbol
       \parfillskip=0pt \finalhyphendemerits=0}
\def\dispqed{\rlap{\qquad\qedsymbol}}

\opn\dis{dis}
\def\pnt{{\raise0.5mm\hbox{\large\bf.}}}

\opn\Lex{Lex}
\opn\int{int}

\newcommand{\inD}[1][\relax]{\def\argone{#1}\def\temprelax{\relax}
  \ifx\argone\temprelax\right.\else\,\middle|#1\right.{}\fi}

\newif\ifbinary
\binarytrue

\begin{document}

\title{Ideals generated by 2-minors, collections of cells and stack polyominoes}

\author{Ayesha Asloob Qureshi}

\address{Ayesha Asloob Qureshi, Abdus Salam School of Mathematical Sciences,
GC University, Lahore.
68-B, New Muslim Town, Lahore 54600, Pakistan} \email{ayesqi@gmail.com}

\begin{abstract}
In this paper we study ideals generated by quite general sets  of $2$-minors of an $m\times n$-matrix of indeterminates. The sets of $2$-minors are defined  by collections of cells and include 2-sided ladders.  For convex collections of cells it is shown that the attached  ideal of $2$-minors is a Cohen--Macaulay prime ideal. Primality is also shown for collections of cells whose connected components are row or column convex. Finally the class group of the ring attached to a stack polyomino and its canonical class is computed, and a classification of the Gorenstein stack polyominoes is given.
\end{abstract}

\maketitle

\section*{Introduction}

Let $K$ be a field and $X=(x_{ij})_{i=1,\ldots,m \atop j=1,\ldots,n}$ be a matrix of indeterminates. In this paper we study ideals generated by quite general sets of 2-minors of $X$. For any integer $1\leq t\leq \min\{m,n\}$, the ideal generated by all   $t$-minors of $X$  is  well understood, see \cite{Ho} and \cite{BV}, and more generally the ideals generated by all $t$-minors of a one and two sided ladders, see for example \cite{A}. Motivated by applications in algebraic statistics, ideals generated by even more general sets of minors have been investigated, including ideals generated by adjacent 2-minors, see \cite{HSS}, \cite{HH} and \cite{OH2}, or ideals generated by an arbitrary set of 2-minors  in an $2 \times n$-matrix \cite{HHH}.

Given an ideal $I$ generated by an arbitrary set of 2-minors of $X$, the question arises when $I$ is a prime or a  radical ideal  and what are its primary components. As shown in \cite{HHH}, $I$ is always radical if $X$ is a $2 \times n$ matrix and the authors give the explicit primary decomposition of such ideals. The problem becomes already much more complicated if $m,n \geq 3$. Easy examples show that $I$ need not to be radical in general.

In this paper we study ideals generated by inner 2-minors of a collection of cells. A cell is a unit square of $\RR^2$ whose corners are elements in $\NN^2$. A collection $\Pc$ of cells is a finite union of cells. We denote by $V(\Pc)$ the set of corners belonging to the cells of $\Pc$. In order to define the ideal of inner 2-minors of a collection of cells $\Pc$ we introduce some terminology. First we introduce the partial order on $\NN^2$ given by $(i,j) \leq (k,l)$ if and only if $i \leq k$ and $j \leq l$. The set $\NN^2$ together with this partial order is a distributive lattice. Let $a,b \in \NN^2$ with $a \leq b$, then the set $[a,b]= \{ c \in \NN^2|\; a \leq c \leq b\}$ is an interval of $\NN^2$. If $a=(i,j)$ and $b=(k,l)$, then the interval $[a,b]$ is called a proper interval if $i<k$ and $j<l$ and the elements $a,b,c,d$ are called the corners of the proper interval $[a,b]$ where $c=(k,j)$ and $d=(i,l)$. In particular, we call $a$, $b$ the diagonal corners of $[a,b]$ and $c$, $d$ the anti-diagonal corners of $[a,b]$. To each collection of cells $\Pc \subset \NN^2$, we attach an ideal $I_{\Pc}$ as follows. Let $K$ be a field and $S$ be the polynomial ring over $K$ in the variables $x_{a}$ with $a \in V(\Pc)$. To each proper interval $[a,b]$ of $\NN^2$, we assign the binomial $f_{a,b}=x_{a} x_{b} - x_{c} x_{d}$, where $c$ and $d$ are the anti-diagonal corners of $[a,b]$. A proper interval $[a,b]$ is called an inner interval of $\Pc$ if all cells of $[a,b]$ belong to $\Pc$. The binomial $f_{a,b}$ is called an inner 2-minor of $\Pc$, if $[a,b]$ is an inner interval of $\Pc$. We denote by  $I_{\Pc} \subset S$ the ideal generated by the inner 2-minors of $\Pc$ and  by $K[\Pc]$ the quotient ring $S/I_{\Pc}$.

The class of ideals attached to a collection of cells includes, for example, the ideals of 2-minors of two sided ladders, but it is much more general. Interesting classes of collections of  cells are the so-called polyominoes that are well studied in various combinatorial contexts.  A collection of cells $\Pc$ is called a polyomino if it is a connected  collection of cells which means that for  any two cells $A,B\in \Pc$ there exists   a sequence of cells $C_1,\ldots, C_m$  with $C_1=A$, $C_m=B$,  and for all $i$, the cells $C_i$ and $C_{i+1}$ have an edge in common.

In Section~\ref{basis} of this paper we introduce some basic concepts related to collection of cells. In particular we introduce column convex, row convex and convex collection of cells. The first main result of this paper is stated in Section~\ref{convexcollection} where it is shown that $K[\Pc]$ is a normal Cohen--Macaulay domain of dimension $|V(\Pc)| - |\Pc|$, if $\Pc$ is convex.

In Section~\ref{embedding}, we define for any collection of cells $\Pc$ a natural toric ring $T_{\Pc}$ and a natural $K$-algebra homomorphism $K[\Pc] \rightarrow T_{\Pc}$. We denote by $\mathfrak{C}$ the class of collect\-ion of cells for which this $K$-algebra homomorphism is an isomorphism. It is shown in Corollary~\ref{prime} that $K[\Pc]$ is domain if and only if $\Pc \in \mathfrak{C}$. We conjecture that $\Pc \in \mathfrak{C}$,  if $\Pc$ is a simple collection of cells. Roughly speaking $\Pc$ is simple if it is connected and has no holes, see Section~\ref{basis} for the precise definition. As a partial result we obtain in Theorem~\ref{colconvex} that a simple collection of cells $\Pc$ belongs to  $\mathfrak{C}$ if each connected component is row or column convex.

As shown in Section~\ref{convexcollection}, $K[\Pc]$ is a normal domain if $\Pc$ is convex, and hence it is of interest to compute the class group of $K[\Pc]$ in this case. In Section~\ref{stackpolyomino}, this is done for a special class of convex collection of cells,  namely for stack polyominoes. In a first step we show in Corollary~\ref{gbasisstcak}, that $I_{\Pc}$ has a quadratic Gr\"obner basis if $\Pc$ is a stack polyomino. Then in Corollary~\ref{classgroup} it is shown that $\Cl(K[\Pc])$ is free. Its rank is determined by the inner corners of $\Pc$. Finally in Theorem~\ref{canonicalclass}, we determine the canonical class of $K[\Pc]$. As a consequence, all Gorenstein stack polyominoes are classified.


\section{Collections of cells}
\label{basis}

In this section we consider collections of cells and  polyominoes to which in the following sections binomial ideals will attached. For this purpose and for later applications we have to introduce some concepts and notation.

We consider on $\NN^2$ the natural partial order defined as follows: $(i,j) \leq (k,l)$ if and only if $i \leq k$ and $j \leq l$.
The set $\NN^2$ together with this partial order is a distributive lattice. Let $a,b \in \NN^2$ with $a \leq b$,
then the set $[a,b]= \{ c \in \NN^2|\; a \leq b \leq c\}$ is an interval of $\NN^2$. If $a=(i,j)$ and $b=(k,l)$,
then the interval $[a,b]$ is called a {\em proper interval} if $i<k$ and $j<l$, and the elements $a,b $ together with the elements $c=(k,j)$ and $d=(i,l)$ are called
the {\em corners} of the proper interval $[a,b]$. The elements  $a$, $b$ are the diagonal corners  and the elements $c$, $d$ the
anti-diagonal corners of $[a,b]$. We say that  $a$ and $b$ are  in {\em horizontal (vertical) position}, if $j=l$ ($i=k$).

The interval $C=[a,b]$ with $b=a+(1,1)$ is called a {\em cell}  of $\NN^2$ (with lower left corner $a$). It may be viewed as a unit square of $\RR^2$ whose corners are positive integer vectors. The elements (corners) of $[a,b]$ are called the {\em vertices} of $C$. We denote the set of vertices of $C$ by $V(C)$. Let $c$, $d$ be the anti-diagonal corners of $C$, then the edges of $C$ are the sets  $\{a,c\}$, $\{a,d\}$, $\{b,c\}$ and $\{b,d\}$. We denote the set of edges of $C$ by $E(C)$.

Let $[a,b]$ be a proper interval in $\NN^2$ with $a=(i,j)$ and $b=(k,l)$. We say a cell $C$ with lower left corner $(r,s)$ belongs to $[a,b]$ if
\begin{eqnarray} \label{inside}
i\leq r \leq k-1 \quad \text{and} \quad j \leq s \leq l-1.
\end{eqnarray}
The cell $C$ is called a {\em border cell} of $[a,b]$ if one of the inequalities in (\ref{inside}) is an equality.

Let $A$ and $B$ be two cells of $\NN^2$ with lower left corners $(i,j)$ and $(k,l)$. Then the {\em cell interval},  denoted by $[A,B]$,  is the set
 \[
 [A,B]=\{E \: \text{$E \in \NN^2$ with lower left corner $(r,s)$, for $i\leq r \leq k$, $ j \leq s \leq l$}\}
 \]
If $(i,j)$ and $(k,l)$ are in horizontal position,  then the cell interval $[A,B]$ is called a {\em horizontal} cell interval. Similarly one defines a vertical cell interval.

\medskip
Let $\Pc$ be a finite collection of cells of $\NN^2$. We set $V(\Pc)=\bigcup_{C \in \Pc} V(C)$ and call it the {\em vertex set} of $\Pc$,  and we set $E(\Pc)=\bigcup_{C \in \Pc} E(C)$ and call it the {\em edge set} of $\Pc$. In this paper, we consider only finite collection of cells of $\NN^2$.

A vertex $a \in V(\Pc)$ is called an {\em interior vertex} of $\Pc$ if $a$ is a vertex of four distinct cells of $\Pc$, otherwise it is called {\em boundary vertex} of $\Pc$. The {\em interior} of $\Pc$ , denoted by $\int(\Pc)$, is the  set of all interior vertices of $\Pc$. The set $\partial \Pc=V(\Pc) \setminus \int(\Pc)$ is called the {\em boundary} of $\Pc$. In Figure~\ref{interior},the fat dots mark the interior vertices of $\Pc$, the other vertices are the boundary vertices of $\Pc$.

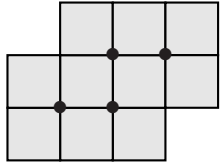
\begin{figure}[hbt]
\begin{center}
\psset{unit=0.7cm}
\begin{pspicture}(-2,1)(0,5)
\pspolygon[style=fyp,fillcolor=light](-1,1)(0,1)(0,2)(-1,2)
\pspolygon[style=fyp,fillcolor=light](-2,1)(-1,1)(-1,2)(-2,2)
\pspolygon[style=fyp,fillcolor=light](-1,2)(0,2)(0,3)(-1,3)
\pspolygon[style=fyp,fillcolor=light](-2,2)(-1,2)(-1,3)(-2,3)
\pspolygon[style=fyp,fillcolor=light](-2,3)(-1,3)(-1,4)(-2,4)
\pspolygon[style=fyp,fillcolor=light](-1,3)(0,3)(0,4)(-1,4)
\pspolygon[style=fyp,fillcolor=light](0,3)(1,3)(1,4)(0,4)
\pspolygon[style=fyp,fillcolor=light](0,2)(1,2)(1,3)(0,3)
\pspolygon[style=fyp,fillcolor=light](-3,1)(-2,1)(-2,2)(-3,2)
\pspolygon[style=fyp,fillcolor=light](-3,2)(-2,2)(-2,3)(-3,3)
\rput(-1,2){$\bullet$}
\rput(-1,3){$\bullet$}
\rput(0,3){$\bullet$}
\rput(-2,2){$\bullet$}
\end{pspicture}
\end{center}
\caption{Interior and boundary of $\Pc$}\label{interior}
\end{figure}

We call $\Pc$ {\em row convex}, if the horizontal cell interval $[A,B]$ is contained in $\Pc$ for any two cells $A$ and $B$ of $\Pc$ whose  lower left corners are in horizontal position. Similarly one defines {\em column convex}. A collection of cells $\Pc$ is called {\em convex} if it is row and column convex.

\medskip
Let $C$ and $D$ be two cells of $\Pc$. Then $C$ and $D$ are {\em connected}, if there is a sequence of cells of $\Pc$ given by $C= C_1, \ldots, C_m =D$ such that $C_i \cap C_{i+1}$ is an edge for $i=1, \ldots, m-1$. If in addition, $C_i \neq C_j$ for all $i \neq j$, then $\mathcal{C}$ is called a {\em path} (connecting $C$ and $D$). The collection of cells $\Pc$ is called a {\em polyomino} if any two cells of $\Pc$ are connected, see Figure~\ref{polyomino}. We notice that each connected component of a finite collection of cells $\Pc$ is a polyomino.

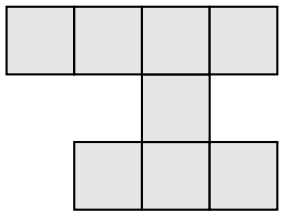
\begin{figure}[hbt]
\begin{center}
\psset{unit=0.9cm}
\begin{pspicture}(4.5,-0.5)(4.5,3.5)
\rput(-1,0){
\pspolygon[style=fyp,fillcolor=light](4,0)(4,1)(5,1)(5,0)
\pspolygon[style=fyp,fillcolor=light](5,0)(5,1)(6,1)(6,0)
\pspolygon[style=fyp,fillcolor=light](3,2)(3,3)(4,3)(4,2)
\pspolygon[style=fyp,fillcolor=light](5,1)(5,2)(6,2)(6,1)
\pspolygon[style=fyp,fillcolor=light](4,2)(4,3)(5,3)(5,2)
\pspolygon[style=fyp,fillcolor=light](5,2)(5,3)(6,3)(6,2)
\pspolygon[style=fyp,fillcolor=light](6,0)(6,1)(7,1)(7,0)
\pspolygon[style=fyp,fillcolor=light](6,2)(6,3)(7,3)(7,2)
}
\end{pspicture}
\end{center}
\caption{A polyomino}\label{polyomino}
\end{figure}

Since $\Pc$ consists of finitely many cells, there exists a proper interval $[a,b] \subset \NN^2$ such that $V(\Pc) \subset \int([a,b])$. The collection of cells $\Pc$ is called {\em simple} if any cell $C$ of $[a,b]$ which does not belong to $\Pc$ is connected  to a border cell $D$ of $[a,b]$ by a path $C=C_1, \ldots, C_m=D $ such that $C_i \notin \Pc$ for all $i = 1, \ldots, m$. Intuitively this means that a simple collection of cells has no holes, see Figure~\ref{simple}.

\begin{figure}[hbt]
\begin{center}
\psset{unit=0.7cm}
\begin{pspicture}(-0.5,0)(-0.5,4)
\rput(-4,0){
\pspolygon[style=fyp,fillcolor=light](-1,1)(0,1)(0,2)(-1,2)
\pspolygon[style=fyp,fillcolor=light](-2,1)(-1,1)(-1,2)(-2,2)
\pspolygon[style=fyp,fillcolor=light](0,1)(1,1)(1,2)(0,2)
\pspolygon[style=fyp,fillcolor=light](-2,2)(-1,2)(-1,3)(-2,3)
\pspolygon[style=fyp,fillcolor=light](-2,3)(-1,3)(-1,4)(-2,4)
\pspolygon[style=fyp,fillcolor=light](-1,3)(0,3)(0,4)(-1,4)
\pspolygon[style=fyp,fillcolor=light](0,3)(1,3)(1,4)(0,4)
\pspolygon[style=fyp,fillcolor=light](0,2)(1,2)(1,3)(0,3)
\rput(-0.5,0){Not simple}
}
\rput(0,2)
{
\pspolygon[style=fyp,fillcolor=light](2.8,0)(2.8,1)(3.8,1)(3.8,0)
\pspolygon[style=fyp,fillcolor=light](3.8,-1)(3.8,0)(4.8,0)(4.8,-1)
\pspolygon[style=fyp,fillcolor=light](4.8,0)(4.8,1)(5.8,1)(5.8,0)
\pspolygon[style=fyp,fillcolor=light](3.8,1)(3.8,2)(4.8,2)(4.8,1)
\pspolygon[style=fyp,fillcolor=light](3.8,0)(3.8,1)(4.8,1)(4.8,0)
\rput(4.2,-2){Simple}
}
\end{pspicture}
\end{center}
\caption{}\label{simple}
\end{figure}

We call $\Pc$ {\em weakly connected} if for any two cells $C$ and $D$ of $\Pc$, there exists a sequence of cells of $\Pc$ given by $C= C_1, \ldots, C_m =D$ such that $C_i \cap C_{i+1} \neq \emptyset$, for $i=1, \ldots, m-1$. Figure~\ref{weaklyconnected} displays a weakly connected collection of cells with two connected components.

\begin{figure}[hbt]
\begin{center}
\psset{unit=0.7cm}
\begin{pspicture}(4.5,0)(4.5,3)
\pspolygon[style=fyp,fillcolor=light](6,0)(6,1)(7,1)(7,0)
\pspolygon[style=fyp,fillcolor=light](7,0)(7,1)(8,1)(8,0)
\pspolygon[style=fyp,fillcolor=light](8,0)(8,1)(9,1)(9,0)
\pspolygon[style=fyp,fillcolor=light](5,1)(5,2)(6,2)(6,1)
\pspolygon[style=fyp,fillcolor=light](4,1)(4,2)(5,2)(5,1)
\pspolygon[style=fyp,fillcolor=light](8,1)(8,2)(9,2)(9,1)
\pspolygon[style=fyp,fillcolor=light](3,1)(3,2)(4,2)(4,1)
\pspolygon[style=fyp,fillcolor=light](2,1)(2,2)(3,2)(3,1)
\pspolygon[style=fyp,fillcolor=light](1,1)(1,2)(2,2)(2,1)
\rput(1.5,1.5){$C$}
\rput(8.5,1.5){$D$}
\end{pspicture}
\end{center}
\caption{A weakly connected collection of cells}\label{weaklyconnected}
\end{figure}
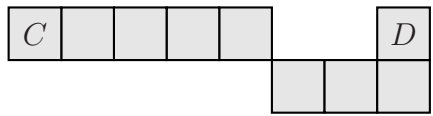

The following  lemmata on collections of cells will be needed in the later sections.

\begin{Lemma}
\label{interval}
Let $\Pc$ be a weakly connected and convex collection of cells, and let $a,b \in V(\Pc)$ be two vertices which are in horizontal or vertical position. Then  $[a,b] \subset V(\Pc)$.
\end{Lemma}

\begin{proof}
Let $a, b \in V(\Pc)$ in horizontal position. We may assume that $|[a,b]| > 2$, otherwise there is nothing to show. There exist two cells $C$ and $D$ in $\Pc$ such that $a$ is a vertex of $C$ and $b$ is a vertex of $D$. The horizontal line $L$ which contains the interval $[a,b]$ divides $\Pc$ in a lower and upper part. If the cells $C$ and $D$ both belong to the upper or to the lower part, then convexity of $\Pc$ gives $[C,D] \subset \Pc$. It shows $[a,b] \subset V(\Pc)$.

Otherwise we may assume that $C$ belongs to the lower part and $D$ belongs to the upper part of $\Pc$. We then use the fact that there exists a sequence of cells $C=C_1, C_2 \ldots, C_r=D$ such that $V(C_i) \cap V(C_{i+1}) \neq \emptyset$ for $i=1,\ldots, r-1$. This sequence has to cross the line $L$, that is, there exists an index $i$ such that $C_i$ belongs to the lower part of $\Pc$ and $C_{i+1}$ belongs to the upper part of $\Pc$. In particular, both $C_i$ and $C_{i+1}$ have an edge whose vertices belong to  $L$. If  $V(C_i)\sect [a,b]\neq \emptyset$ or $V(C_{i+1})\sect [a,b]\neq \emptyset$, then there exists $c\in [a,b]$ with  $c\neq a,b$ which belongs to $V(\Pc)$. Induction on the length of the interval, concludes the proof in this case. Otherwise, by using convexity of $\Pc$ we see that $[C_i, C]$ and $[D, C_{i+1}]$ are horizontal cell intervals of $\Pc$ such that either $[a,b] \subset V([C_i, C])$ or $[a,b] \subset V([D, C_{i+1}])$. This completes the proof.

The arguments are similar for the case when $a$ and $b$ are in vertical position.
\end{proof}

\begin{Lemma}
\label{corners}
Let $\Pc$ be a weakly connected and convex, and $[g,h]$ be a proper interval in $\NN^2$. If the corners of $[g,h]$ belong to $V(\Pc)$, then the cells of $[g,h]$ belong to $\Pc$.
\end{Lemma}

\begin{proof}
It is clear by Lemma~\ref{interval} that if the corners of $[g,h]$ belong to $V(\Pc)$, then $[g,h] \subset V(\Pc)$. Suppose that there exists a cell $E$ of $[g,h]$ which does not belong to $\Pc$. Let $a=(i,j)$, $b=(i+1,j)$, $c=(i+1,j+1)$, $d=(i,j+1)$ be the vertices of $E$. Since these vertices belong to $V(\Pc)$, there exist cells $A,B,C,D$ in $\Pc$ such that $a \in A$, $b \in B$, $c \in C$ and $d \in D$. If two of these cells are in horizontal or vertical positions then from the fact that $\Pc$ is convex one easily deduces that $E \in \Pc$. Otherwise, up to rotation, the only possible configuration of the cells $A,B,C,D$ is shown in Figure~\ref{ABCD}.

\begin{figure}[hbt]
\begin{center}
\psset{unit=1cm}
\begin{pspicture}(4.5,0.5)(4.5,2.5)
\pspolygon[style=fyp,fillcolor=light](4,2)(4,3)(5,3)(5,2)
\pspolygon[style=fyp,fillcolor=light](3,0)(3,1)(4,1)(4,0)
\pspolygon[style=fyp,fillcolor=light](5,1)(5,2)(6,2)(6,1)
\rput(3.5,0.5){A}
\rput(5.5,1.5){B}
\rput(4.5,2.5){D=C}
\rput(4,1){$\bullet$}
\rput(4,2){$\bullet$}
\rput(5,2){$\bullet$}
\rput(5,1){$\bullet$}
\rput(3.8,1.2){a}
\rput(3.8,2.2){d}
\rput(5.2,2.2){c}
\rput(5.2,0.8){b}

\end{pspicture}
\end{center}
\caption{}\label{ABCD}
\end{figure}

By using the assumption that $\Pc$ is weakly connected there exists a sequence of cells $A=F_1, F_2 \ldots, F_m=C$ such that $V(F_i) \cap V(F_{i+1}) \neq \emptyset$ for $i=1,\ldots, m-1$. It implies that there exists at least one $F_i$ such that $[F_i, C]$ is a vertical cell interval of $\Pc$ or $[F_i, B]$ is a horizontal cell interval of $\Pc$. Again, by using the fact that $\Pc$ is convex, we have $E \in P$, a contradiction.
\end{proof}

\begin{Lemma}
Let $\Pc$ be a simple collection of cells and $\Pc_1$ and $\Pc_2$ be two connected components of $\Pc$. Then $|\Pc_1\cap \Pc_2| \leq 1$.
\end{Lemma}
\begin{proof}
Let $[a,b] \subset \NN^2$ such that $V(\Pc) \subset [a,b]$. We may assume that  $\Pc_1$ and $\Pc_2$ meet at least at one point, say $p$. Then there exist $C \in \Pc_1$ and $D \in \Pc_2$ such that $C \cap D = \{p\}$, and two distinct uniquely determined cells $E$ and $F$ in $[a,b]$ not belonging to  $\Pc$ such that $p$ is a vertex of $E$ and $F$. Since $\Pc$ is simple, each of the cells $E$ and $F$ are connected to a border cell of $[a,b]$ by the paths $\Ec\: E_1, \ldots, E_r$ and $\Fc\: F_1, \ldots, F_s$, respectively, where each $E_i$ and $F_j$ do not belong to $\Pc$.

Let $\Rc$ be the collection of cells of $[a,b]$ and $\Qc=\Rc \setminus  \Ec \cup \Fc $. If $\Ec \cap \Fc = \emptyset$, then $\Qc$ consists of two connected components $\Qc_1$ and $\Qc_2$ such that $V(\Qc_1) \cap V(\Qc_2)= p$. Let $C \in \Qc_1$ and $D \in \Qc_2$. Then $\Pc_1 \subset \Qc_1$ and $\Pc_2 \subset \Qc_2$, because $\Pc_1$ and $\Pc_2$ are connected components of $\Pc$. Hence $|\Pc_1 \cap \Pc_2| = 1$

If $\Ec \cap \Fc \neq \emptyset$, then let $i$ and $j$ be the smallest integer such that $E_i = F_j$. We can replace by $\Ec$ by the path $\Ec'= E_1, \ldots, E_i, F_{j+1}, \ldots, F_s$ that connects $E$ to a border cell of $[a,b]$. Then again, by letting  $\Qc= \Rc \setminus \Ec \cup \Fc$,  we obtain the desired conclusion.
\end{proof}

\noindent
Let $\Pc$ be a weakly connected collection of cells with connected components $\Pc_1, \ldots, \Pc_r$.  We assign to $\Pc$ a graph $G$  with vertex set $V(G)=[r]$ and edge set $E(G)$  as follows: $\{i,j\} \in E(G)$ is and only if $\Pc_i \cap \Pc_j \neq \emptyset$.

\begin{Lemma}
\label{tree}
Let $\Pc$ be a weakly connected, simple collection of cells. Then the graph $G$ attached to $\Pc$ is a tree.
\end{Lemma}
\begin{proof}
Let $\Pc_1, \ldots, \Pc_m$ be the connected components of $\Pc$. Suppose that the graph $G$ attached to $\Pc$ is not a tree. Then $G$ contains a cycle $W$ with no chords. We may assume that
$E(W)= \{r,1\} \cup \{ \{i, i+1\}: i=1, \ldots, r-1   \} $. Let $\Pc'=\Pc_1 \cup \ldots \cup \Pc_r$, and $[a,b]\subset \NN^2$
be an interval containing $V(\Pc)$. First we show that $\Pc'$ is also simple.
Let $C \in \Rc \setminus \Pc'$. If $C \notin \Pc$ then $C$ can be connected to a border cell of $[a,b]$ by a path of cells
outside $\Pc$ which are also outside $\Pc'$.  Now suppose that $C \in \Pc$. Then $C \in \Pc_j$ for some
$j \not\in \{1,\ldots,r\}$. Let $D \in \Rc \setminus \Pc$ such that $D$ has a
common edge with a cell, say $A$, of $\Pc_j$. Since $\Pc_j$ is connected, $C$ can be connected to $A$ by a path $\mathcal{E}$ of cells in $\Pc_j$.
Let $\mathcal{F}$ be a path of cells outside $\Pc$ which connects $D$ to a border cell of $[a,b]$. Such  a path exists because $\Pc$ is simple. By adjoining $\mathcal{E}$, $D$ and $\mathcal{F}$,
we obtain a path of cells outside $\Pc'$ that connects $C$ to a border cell of $[a,b]$. Thus any cell in $[a,b]$ which does not belong to $\Pc'$
can be connected to a border cell of $[a,b]$. This shows that  $\Pc'$ is simple.

Let $\Rc$ be the  collection of cells of $[a,b]$. Then $\Rc \setminus \Pc' = \Qc_1 \cup \Qc_2$ such that $\Qc_1$ and $\Qc_2$ are not connected. Choosing $[a,b]$ large enough  we have that $\Qc_1$ and $\Qc_2$  are non-empty. Let $A \in \Qc_1$ and $B \in \Qc_2$. Since $\Pc'$ is simple, the cells  $A$ and $B$ are connected to  border cells of $[a,b]$ by  paths whose cells do not belong to $\Pc'$. Choosing $[a,b]$ even bigger if needed, these two border cells can be connected by a path whose cells also do not belong to $\Pc'$.  It follows that $\Qc_1$ and $\Qc_2$ are connected, a contradiction.
\end{proof}

\section{Convex collections of cells  and inner minors}
\label{convexcollection}

Let $\Pc \subset \NN^2$ be a collection of cells.  We attach to $\Pc$  an ideal $I_{\Pc}$ as follows. Let $K$ be a field and $S$ the  polynomial ring over $K$ in the variables $x_{a}$ with $a \in V(\Pc)$. To each proper interval $[a,b]$ of $\NN^2$, we assign the binomial $f_{a,b}=x_{b} x_{a} - x_{c} x_{d}$, where $c$ and $d$ are the anti-diagonals corners of $[a,b]$. A proper interval $[a,b]$ is called an {\em inner interval} of $\Pc$ if all cells of $[a,b]$ belong to $\Pc$. The binomial $f_{a,b}$ is called an {\em inner 2-minor} of $\Pc$, if $[a,b]$ is an inner interval of $\Pc$. Then  $I_{\Pc} \subset S$ be the ideal generated by inner 2-minors of $\Pc$. We denote by $K[\Pc]$ the quotient ring $S/I_{\Pc}$.

We will compare $I_{\Pc}$ with a toric ideal which is naturally given by $\Pc$. Let  $[a,b] \subset \NN^2$ be the smallest interval which contains $V(\Pc)$. After a shift of coordinates, we may assume that $a=(1,1)$ and $b=(m,n)$. To $\Pc$ we attach the toric ring $R=K[s_i t_j| \; (i,j)\in V(\Pc)] \subset K[s_1, \ldots, s_m, t_1, \ldots, t_n]$. We associate a bipartite graph $G$ with vertex set $\{s_1, \ldots , s_m\} \cup \{t_1, \ldots, t_n\}$ to $\Pc$ such that each vertex $(i,j) \in V(\Pc)$ determines the  edge $\{s_i, t_j\}$ in $G$. The toric ring $R$ can then be viewed as the edge ring of $G$. For the sake of convenience, in this section we denote for $a=(i,j) \in V(\Pc)$ the variable  $x_a$ in $S$ by $x_{ij}$.

A cycle $w$ of $G$ is a subset $\{s_{i_1}, t_{j_1}, s_{i_2}, t_{j_2}, \ldots , s_{i_{r-1}}, t_{j_{r-1}}, s_{i_r}, t_{j_r}\}$ of the vertex set of $G$ such that for $k= 1, \ldots, r$, each $\{s_{i_k}, t_{j_k}\}$ and $\{t_{j_k}, s_{j_{k+1}}\}$ is an edge of $G$, where $i_{r+1} = i_1$. To each such cycle $w$ we associate the binomial $f_w=x_{i_1j_1} x_{i_2j_2}\ldots x_{i_{r-1}j_{r-1}}x_{i_rj_r} - x_{i_2j_1}x_{i_3j_2}\ldots x_{i_rj_{r-1}}x_{i_1j_r}$. Observe that a binomial $f$ is attached to a cycle of length $4$ if and only if $f=f_{a,b}$ where $[a,b]$ is a proper interval of $V(\Pc)$.

Let $\phi: S \rightarrow R$ be the $K$-algebra homomorphism defined by $\phi(x_{ij})=s_i t_j$, for all $(i,j) \in V(\Pc)$ and set $J_{\Pc}=\Ker \phi$. It is known, see  \cite[Lemma 1.1]{OH} and \cite[Proposition 8.1.2]{V}, that the kernel $J_{\Pc}$ of $\varphi$ is generated by the binomial $f_w$, where $w$ is an (even) cycle of $G$. Since each generator of $I_{\Pc}$ corresponds to a cycle of length 4, we have $I_{\Pc} \subset J_{\Pc}$.

\begin{Theorem}
\label{graphprime}
Let $\Pc$ be a collection of cells.  Then the following holds:
\begin{enumerate}
\item[{\em (a)}] If $I_{\Pc} = J_{\Pc}$, then $\Pc$ is convex.
\item[{\em (b)}] If $\Pc$ is convex and weakly connected, then $I_{\Pc} = J_{\Pc}$.
\end{enumerate}
\end{Theorem}

\begin{proof}
(a) Suppose that we have the equality $I_{\Pc} = J_{\Pc}$. Let $C$ and $D$ be two cells of $\Pc$ with lower left corner $a=(i,j)$ and $b=(k,j)$ with $i<k$. Then the corners of the interval $[a,e]$ belong to $V(\Pc)$, where $e=(k+1,j+1)$. Therefore, the binomial $f_{a,e}$ belongs to $J_{\Pc}$, and hence $f_{a,e} \in I_{\Pc}$. It shows that $f_{a,e}$ is a linear combination of inner 2-minors of $\Pc$. Thus there is an inner 2-minor $f_{g,h}$ of $\Pc$ which contains the term $x_{ij}x_{k+1,j+1}$. This is possible if and only if $g=a$ and $h=e$. This implies that $[a,e]$ is an inner interval of $\Pc$. Hence $\Pc$ is row convex. Similarly one shows that $\Pc$ is column convex and hence $\Pc$ is convex.

(b) Suppose that $\Pc$ is convex and weakly connected. First observe that each cycle $v=\{s_i, t_j, s_k, t_l\}$ with $i<k$ and $j<l$ of length 4 in $G$ determines the four vertices $(i,j), (k,l), (k,j), (i,l)$ of $V(\Pc)$. It follows by Lemma~\ref{corners}, that the cells of $[(i,j),(k,l)]$ belong to $\Pc$. In other words, any binomial $f_{a,b}$ with $[a,b] \subset V(\Pc)$ is an inner 2-minor of $\Pc$. With this observation it suffices to show that for a cycle $w$ of length $2r$ with $r > 2$ of $G$, the associated binomial $f_w$ can be written as a linear combination of binomials $f_p$ and $f_q$, where $q$ and $p$ are cycles of $G$ of length 4 and $2(r-1)$ respectively.

Let $w$ be a cycle of $G$ of length $2r$ with $r \geq 3$ given by
\[
\{s_{i_1}, t_{j_1}, s_{i_2}, t_{j_2},\ldots s_{i_{r-1}}, t_{j_{r-1}}, s_{i_r}, t_{j_r}\}
\]
and let $f_w=x_{i_1j_1}x_{i_2j_2}\ldots x_{i_{r-1}j_{r-1}}x_{i_rj_r} - x_{i_2j_1}x_{i_3j_2}\ldots x_{i_rj_{r-1}}x_{i_1j_r}$ be its associated binomial in $J_{\Pc}$.  Moreover, we may assume that $i_1 \leq i_k$ for all $k$.

Assume $i_2 > i_r$. Then Lemma~\ref{interval} implies that $x_{{i_r}{j_1}}\in V(\Pc)$, because $x_{i_1j_1}$ and $x_{i_2j_1}$ belong to $V(\Pc)$ (horizontal position). Take
\[
q=\{s_{i_1}, t_{j_1}, s_{i_r}, t_{j_r}\}
\]
and
\[
p=\{s_{i_r}, t_{j_1}, s_{i_2}, t_{j_2}, \ldots, s_{i_{r-1}}, t_{j_{r-1}}\}
\]
with associated binomials $f_q=x_{i_1j_1}x_{i_rj_r}-x_{i_rj_1}x_{i_1j_r}$ and $f_p=x_{i_rj_1}x_{i_2j_2}\ldots x_{i_{r-1}j_{r-1}}- x_{i_2j_1}x_{i_3j_2}\ldots x_{i_rj_{r-1}}$, respectively. Then $f_w= x_{i_2j_2}\ldots x_{i_{r-1}j_{r-1}} f_q + x_{i_1j_r} f_p $, as required.

Now assume $i_r > i_2$. Applying again Lemma~\ref{interval} we see that $x_{i_2j_r}\in V(\Pc)$, because $x_{i_1j_r}$ and $x_{i_rj_r}$ belong to $V(\Pc)$ (horizontal position). Take
\[
q=\{s_{i_1}, t_{j_1}, s_{i_2}, t_{j_r}\}
\]
and
\[
p=\{s_{i_2}, t_{j_2}, s_{i_3}, t_{j_3}, \ldots, s_{i_r}, t_{j_r}\}
\]
with associated binomials $f_q=x_{i_1j_1}x_{i_2j_r}-x_{i_2j_1}x_{i_1j_r}$ and $f_p=x_{i_2j_2}x_{i_3j_3}\ldots x_{i_rj_r}- x_{i_3j_2}x_{i_4j_3}\ldots x_{i_rj_{r-1}}x_{i_2j_r}$ Then $f_w= x_{i_1j_1} f_p + x_{i_3j_2}x_{i_4j_3}\ldots x_{i_rj_{r-1}} f_q$, as required.
\end{proof}

In order to formulate the main result of this section we introduce the following definition. Let $[a,b]$ be an interval in $\NN^2$ with $a=(i,j)$ and $b=(k,l)$. Then the {\em size} of $[a,b]$ is defined to be the number $k+l - (i+j)$ and denoted by $\size([a,b])$.

\begin{Theorem}
\label{convexdimension}
Let $\Pc$ be a convex collection of cells. Then $K[\Pc]$ is a normal Cohen--Macaulay domain of dimension $|V(\Pc)|-|\Pc|$. In particular, if $\Pc$ is weakly connected and $[a,b] \subset \NN^2$ is the smallest interval with the property that $V(\Pc) \subset [a,b]$. Then $K[\Pc]$ is a Cohen--Macaulay domain with $\dim K[\Pc]= \size([a,b]) + 1$.
\end{Theorem}

\begin{proof}
Let $\Pc_1, \ldots, \Pc_r$ be the weakly connected component of $\Pc$. Then $V(\Pc)$ is the  disjoint union of the $V(\Pc_j)$, $j=i,\ldots,r$, and $I_{\Pc}= \sum_{j=1}^{r} I_{\Pc_j}$. It follows that $K[\Pc]$ is a normal Cohen--Macaulay domain if and only if each $K[\Pc_j]$ is a normal Cohen--Macaulay domain. Hence we may assume that the $\Pc$ is weakly connected. It follows from Theorem~\ref{graphprime} that $I_{\Pc} = J_{\Pc}$. This implies that $K[\Pc]$ is a domain. We know from \cite[Lemma 1.1]{OH} and \cite[Proposition 8.1.2]{V} that binomials corresponding to the even cycles of the graph $G$ attached to $\Pc$ form the universal Gr\"obner basis of $J_{\Pc}$. This implies that the initial ideal of $I_{\Pc}$ with respect to any monomial order is squarefree. By theorem of Sturmfels \cite{St}, one obtains that $K[\Pc]$ is normal and by a theorem of Hochster \cite[Theorem 6.3.5]{BH} (see also \cite{BH}), we obtain that $K[\Pc]$ is Cohen--Macaulay.

For the computation of the dimension of $K[\Pc]$, we may again assume that $\Pc$ is weakly connected because $|V(\Pc)| - |P| = \sum_{j=1}^{r}(|V(\Pc_j)|-|P_j|)$. Since $K[\Pc]$ is isomorphic to $K[G]$, the edge ring of the bipartite graph $G$, we may apply the  \cite[Corollary 8.2.13]{V} of Villarreal, which says that if $G$ is a connected bipartite graph then $\dim K[G] = |V(G)| -1$. For simplicity we may assume that the smallest interval with the property that $V(\Pc) \subset [a,b]$ is given by $a=(1,1)$ and $b=(m,n)$. Then $[a,b] = [m] \times [n]$. It follows from the identification of $K[G]$ with $K[\Pc]$ that $V(G)= \{s_1, \ldots, s_m\} \cup \{t_1, \ldots, t_n\}$. Therefore, $\dim K[\Pc]= \size([a,b]) + 1$.

It remains to show that $|V(\Pc)| - |P|= \size([a,b]) + 1$. We prove this by induction on the number of columns of $\Pc$. If $\Pc$ consists of only one column then the assertion is trivial. Now assume that number of columns of $\Pc$ is bigger than one, and let $\mathcal{Q}$ be the collection of cells which is obtained from $\Pc$ by removing the right most column $S$ of $\Pc$. Let $[a', b']$ be the smallest interval containing $V(\mathcal{Q})$, and $t$ be the number of cells in $S$ which have a common edge with a cell in $\mathcal{Q}$ and let $r$ be the number of the remaining cells in $S$. Then $|V(\Pc)|= |V(\mathcal{Q})| + 2r +t+1$ and $|\Pc| = |\mathcal{Q}| + r +t$, and $\size([a,b])= \size ([a', b']) + r+1$. Hence we obtain the desired formula.
\end{proof}
As an immediate consequence of Theorem~\ref{convexdimension} we get
\begin{Corollary}
Let $\Pc$ be a convex collection of cells. Then $\height I_{\Pc}= |\Pc|$.
\end{Corollary}

\section{A natural toric ring associated with  a collection of cells}
\label{embedding}
Let $\Pc$ be a collection of cells. An element of $V(\Pc)$ is called a {\em free vertex} if it is not a lower left corner of any cell of $\Pc$ and we denote the set of free vertices of $\Pc$ by $F(\Pc)$.

Let $K$ be a field and as before $S=K[x_a \: a \in V(\Pc)]$. Consider the Laurent polynomial ring $T=K[y_c^{\pm 1} \: c \in F(\Pc)]$. We define a $K$-algebra homomorphism $\psi : S \rightarrow T$ by $x_a \mapsto u_{a}$, where $u_{a}$ is a monomial in $T$. The monomials $u_{a}$ are recursively defined as follows: For each free vertex $ a \in F(\Pc)$, we set $u_{a}=y_{a}$. Let $k=\max \{|a| : \; a=(i,j) \in V(\Pc)\}$ where $|a|=i+j$ for $a=(i,j)$. If $|a|=k$, then $a$ is a free vertex in $\Pc$ and $u_{a}$ is already defined. Suppose now that $|a| < k$ and $a$ is not a free vertex. Then $a=(i,j)$ is the lower left corner of the (unique) cell whose other  vertices are $b=(i+1,j)$, $c=(i,j+1)$ and $d=(i+1,j+1)$. In this case we set $u_{a}=u_b u_c u_d^{- 1}$. Observe that for any $a \in V(\Pc)$ and $y_b^{\pm 1} \in \supp(u_a)$, we have $a\leq b$.

The image of $\psi$ is a toric ring and we set $L_{\Pc} = \Ker \psi$. We denote by $\mathfrak{C}$ the class of collection of cells for which $\psi$ is an isomorphism.
\begin{Lemma}
\label{psi}
Let $[a,b]$ be an inner interval of $\Pc$. Then $u_{a}=u_{c} u_{d} u_{b}^{- 1}$, where $c$ and $d$ are the anti-diagonals of $[a,b]$.
\end{Lemma}
\begin{proof}
We apply induction on the size of the inner interval $[a,b]$ of $V(\Pc)$. The smallest possible size of an inner interval is 2, in which case $[a,b]$ is a cell. Then the assertion follows from the definition of the monomials $u_a$. Let $\size([a,b])>2$, and $a=(i,j)$ and $b=(k,l)$. Then we may assume that $k>i+1$. Let $e=(k-1,j)$ and $f=(k-1,l)$, then $[a,f]$ and $[e,b]$ are two inner interval of $V(\Pc)$ of smaller size than the interval $[a,b]$. Therefore by induction hypothesis we have $u_{a}=u_{c} u_{e} u_{f}^{- 1}$ and $u_{e}=u_{d} u_{f} u_{b}^{- 1}$. Substituting the second formula into the first one we get desired result.
\end{proof}
\begin{Theorem}
\label{contained}
With the notation introduced we have, $I_{\Pc} \subset L_{\Pc} \subset J_{\Pc}$. Moreover, the following cases are possible.
\begin{enumerate}
\item[{\em (1)}] $I_{\Pc} = L_{\Pc} \subsetneq J_{\Pc}$
\item[{\em (2)}] $I_{\Pc} \subsetneq L_{\Pc} \subsetneq J_{\Pc}$
\item[{\em (3)}] $L_{\Pc} = J_{\Pc}$.
\end{enumerate}
If $L_{\Pc} = J_{\Pc}$, then $\Pc$ is convex. In addition, if $\Pc$ is weakly connected then $I_{\Pc}= J_{\Pc}$.
\end{Theorem}
\begin{proof}
Let $[a,b]$ be an inner interval of $\Pc$ with anti-diagonal corners $c$ and $d$, and $f_{a,b}=x_{b} x_{a} - x_{c} x_{d}$ be the corresponding generator in $I_{\Pc}$. By Lemma~\ref{psi}, we have $u_{a} u_{b} = u_{c} u_{d} $. From this it follows that $f_{a,b} \in L_{\Pc}$. Hence $I_{\Pc} \subset L_{\Pc}$.

In order to show that  $L_{\Pc} \subset J_{\Pc}$, we let $W=K[\{s_i^{\pm 1}, t_j^{\pm i} \} : (i,j) \in V(\Pc)]$ and define the $K$-algebra homomorphism $\alpha : T \rightarrow W$ by $\alpha (y_c^{\pm 1})= (s_i t_j)^{\pm 1}$ where $c=(i,j)$. Let $\phi : S \to R$ be the $K$-algebra homomorphism as defined in  before Theorem~\ref{graphprime}. For simplicity we again denote by $\phi$ the composition of $\phi$ with the natural inclusion of $R$ into $W$. We claim that $\phi = \alpha \circ \psi$. This claim will imply that $L_{\Pc} \subset J_{\Pc}$.

In order to prove the claim, let $a=(i,j) \in V(\Pc)$. If $a$ is a free vertex of $V(\Pc)$, then
\[
\alpha \circ \psi (x_{a}) = \alpha (y_a) = s_i t_j = \phi(x_{a}).
\]
Let $k=\max \{|a| : \; a=(i,j) \in V(\Pc)\}$. If $|a|=k$, then $a$ is a free vertex, and the assertion is true as we have just seen. Suppose now that $|a| < k$ and $a$ is not a free vertex. Then $a$ is the lower left corner of the (unique) cell with  vertices $b=(i+1,j)$, $c=(i,j+1)$ and $d=(i+1,j+1)$, and $u_{a}=u_b u_c u_d^{- 1}$. We may assume that for any $e=(p,q)$ with $|e| > |a|$ we have $\alpha (u_e)=\alpha \circ \psi (x_{e}) =\phi(x_{e})= s_p t_q$. Then
\begin{eqnarray*}
\alpha \circ \psi (x_{a}) &=& \alpha (u_a) = \alpha (u_b u_c u_d^{- 1})= \alpha (u_b) \alpha ( u_c) \alpha ( u_d^{- 1}) \\
&=& s_{i+1} t_{j} s_{i} t_{j+1} s_{i+1}^{-1} t_{j+1}^{-1} = s_i t_j =  \phi(x_{a}).
\end{eqnarray*}

Case (1) happens for example when we let $\Pc$ be the collection of cells given in Figure~\ref{IP=LP}.

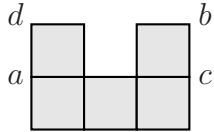
\begin{figure}[hbt]
\begin{center}
\psset{unit=0.7cm}
\begin{pspicture}(4.5,0)(4.5,2)

\pspolygon[style=fyp,fillcolor=light](2.8,0)(2.8,1)(3.8,1)(3.8,0)
\pspolygon[style=fyp,fillcolor=light](3.8,0)(3.8,1)(4.8,1)(4.8,0)
\pspolygon[style=fyp,fillcolor=light](4.8,0)(4.8,1)(5.8,1)(5.8,0)
\pspolygon[style=fyp,fillcolor=light](2.8,1)(2.8,2)(3.8,2)(3.8,1)
\pspolygon[style=fyp,fillcolor=light](4.8,1)(4.8,2)(5.8,2)(5.8,1)

\rput(2.5,1){$a$}
\rput(6.1,1){$c$}
\rput(2.5,2.2){$d$}
\rput(6.1,2.2){$b$}
\end{pspicture}
\end{center}
\caption{$I_{\Pc} = L_{\Pc} \subsetneq J_{\Pc}$}\label{IP=LP}
\end{figure}

The binomial $f_{a,b}$ belongs to $J_{\Pc}$ but not to $L_{\Pc}$. Also $I_{\Pc} = L_{\Pc}$ because of Corollary~\ref{prime} and Theorem~\ref{colconvex}.

Case(2) happens for example when we let $\Pc$ be the collection of cells given in Figure~\ref{chessboard}.
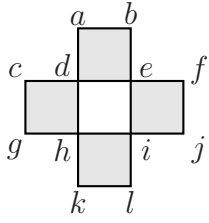
\begin{figure}[hbt]
\begin{center}
\psset{unit=0.7cm}
\begin{pspicture}(4.5,-1)(4.5,2)
\pspolygon[style=fyp,fillcolor=light](2.8,0)(2.8,1)(3.8,1)(3.8,0)
\pspolygon[style=fyp,fillcolor=light](3.8,-1)(3.8,0)(4.8,0)(4.8,-1)
\pspolygon[style=fyp,fillcolor=light](4.8,0)(4.8,1)(5.8,1)(5.8,0)
\pspolygon[style=fyp,fillcolor=light](3.8,1)(3.8,2)(4.8,2)(4.8,1)
\rput(3.8,2.2){$a$}
\rput(4.8,2.3){$b$}
\rput(2.6,1.2){$c$}
\rput(3.5,1.3){$d$}
\rput(5.1,1.2){$e$}
\rput(6.1,1.2){$f$}
\rput(2.6,-0.3){$g$}
\rput(3.5,-0.3){$h$}
\rput(5.1,-0.3){$i$}
\rput(6.1,-0.3){$j$}
\rput(3.8,-1.3){$k$}
\rput(4.8,-1.3){$l$}
\end{pspicture}
\end{center}
\caption{$I_{\Pc} \subsetneq L_{\Pc} \subsetneq J_{\Pc}$}\label{chessboard}
\end{figure}

The binomial $x_a x_f x_g x_l - x_b x_c x_j x_k$ belongs to $L_{\Pc}$ but not to $I_{\Pc}$, and the binomial $ x_e x_h - x_d x_i $ belongs to $J_{\Pc}$ but not to $L_{\Pc}$.

Case (3) happens for example when we let $\Pc$ be the collection of cells given in Figure~\ref{Lp=Jp}.

\begin{figure}[hbt]
\begin{center}
\psset{unit=0.7cm}
\begin{pspicture}(4.5,-1)(4.5,2)
\pspolygon[style=fyp,fillcolor=light](2.8,0)(2.8,1)(3.8,1)(3.8,0)
\pspolygon[style=fyp,fillcolor=light](3.8,0)(3.8,1)(4.8,1)(4.8,0)
\pspolygon[style=fyp,fillcolor=light](3.8,-1)(3.8,0)(4.8,0)(4.8,-1)
\pspolygon[style=fyp,fillcolor=light](4.8,0)(4.8,1)(5.8,1)(5.8,0)
\pspolygon[style=fyp,fillcolor=light](3.8,1)(3.8,2)(4.8,2)(4.8,1)
\end{pspicture}
\end{center}
\caption{$L_{\Pc} = J_{\Pc}$}\label{Lp=Jp}
\end{figure}

Now suppose that $L_{\Pc} = J_{\Pc}$. First observe that by Theorem~\ref{degree}, the  second degree component of $I_{\Pc}$ and $L_{\Pc}$ are equal. Therefore, our assumption implies that second degree component of $I_{\Pc}$ and $J_{\Pc}$ are equal as well. Hence as in the proof of Theorem~\ref{graphprime} (a) we conclude $\Pc$ is convex. Let $\Pc_1, \ldots, \Pc_r$ be the weakly connected components of $\Pc$. Then it follows from the definition of $L_{\Pc}$ that $L_{\Pc} = \sum_{i=1}^{r} L_{\Pc_i}$. If in addition $\Pc$ is weakly connected, we apply Theorem~\ref{graphprime} (b) and obtain the equality $I_{\Pc} = L_{\Pc} = J_{\Pc}$.
\end{proof}

We do not know of any example for which $I_{\Pc} \subsetneq L_{\Pc} = J_{\Pc}$.

\medskip
For the proof of the next theorem we need the following lemma.

\begin{Lemma}
\label{inner}
Let $\Pc$ be a collection of cells. If $f_{a,b} \in L_{\Pc} $, then $[a,b]$ is an inner interval of $\Pc$.
\end{Lemma}

\begin{proof}
Let $f_{a,b}=x_{b}x_{a} - x_{c} x_{d}$ where $c$ and $d$ are the anti-diagonal corners of $[a,b]$. The vertex $a \in V(\Pc)$ is not a free vertex, otherwise $u_a = y_a$ and $y_a^{\pm 1}\notin \supp(u_c u _d)$ and $y_a^{-1} \notin \supp(u_b)$. This  implies that $f_{a,b} \notin L_{\Pc}$, a contradiction.

Assume that $[a,b]$ is not an inner interval. Then there exists an inner interval $[a,e]$ of $V(\Pc)$ with $e < b$ such that $e$ is a free vertex in $\Pc$ and $y_e^{\pm 1} \notin \supp(u_b)$. By Lemma~\ref{psi}, we have $y_e^{\pm 1} \in \supp(u_a)$. On the other hand, since $c \nleq e$, $d \nleq e$, it follows that $y_e^{\pm 1} \notin \supp(u_c u_d)$, contradicting the fact that $u_au_b=u_cu_d$.
\end{proof}

Let $I$ be a graded ideal. The $k^{th}$ graded component of $I$ will be denoted by $I_k$.

\begin{Theorem}
\label{degree}
Let $\Pc$ be a collection of cells. Then $(I_{\Pc} )_2=(L_{\Pc} )_2$.
\end{Theorem}

\begin{proof}
Let $f \in (L_{\Pc} )_2$. By Theorem~\ref{contained}, we have $f \in J_{\Pc}$. It shows that $f$ is a binomial associated to a cycle of length $4$ in the bipartite graph $G$ attached to $\Pc$, or equivalently, $f=f_{a,b}$, where $[a,b]$ is a proper interval of $V(\Pc)$. By Lemma~\ref{inner}, we obtain that $[a,b]$ is an inner interval of $\Pc$. Therefore, $f \in (I_{\Pc})_2$. Hence $(L_{\Pc} )_2 \subset (I_{\Pc} )_2$. The other inclusion follows from Theorem~\ref{contained}.
\end{proof}

We shall need some concepts related to lattice ideals. Let $\Lambda \subset Z^n$ be a lattice. Let $K$ be a field. The lattice ideal attached to $\Lambda$ is the binomial ideal $I_{\Lambda} \subset  K[x_1,\ldots , x_n]$ generated by all binomials
\[
x^a - x^b \quad \text{with} \quad  a- b \in \Lambda \quad \text{and} \quad a,b \in \NN^n.
\]
$\Lambda$ is called saturated if for all $a\in \ZZ^n$ and $c \in \ZZ$ such that $ca \in \Lambda$ it follows that $a\in \Lambda$. The lattice ideal $I_{\Lambda}$ is a prime ideal if and only if $\Lambda$ is saturated. Let $v_1,\ldots , v_m$ be a basis of $\Lambda$. Hosten and Shapiro \cite{HS} call the ideal generated by
the binomials $x^{v_i^+} -  x^{v_i^-}$, $i = 1, \ldots ,m$, a {\em lattice basis ideal }of $\Lambda$. Here $v^+$ denotes the vector obtained from $v$ by replacing all negative components of $v$ by zero, and $v^- = -(v - v^+)$. It is known from \cite{ES} that the ideal generated by all adjacent $2$-minors of an $m \times n$ matrix $X$ of indeterminates is a lattice basis ideal, and that the corresponding lattice ideal is just the ideal of all $2$-minors of $X$. It follows that an ideal which is generated by any set of adjacent $2$-minors of $X$ is again a lattice basis ideal and that its corresponding lattice $\Lambda$ is saturated. Therefore its lattice ideal $I_{\Lambda}$ is a prime ideal.

\begin{Theorem}
\label{latticeideal}
Let $\Pc$ be a collection of cells. Then there exists a saturated lattice $\Lambda$ such that $L_{\Pc} = I_{\Lambda}$.
\end{Theorem}
\begin{proof}
Let $[a,b]$ be the smallest proper interval of $\NN^2$ which contains $V(\Pc)$. After a shift of coordinates, we may assume that $a=(1,1)$ and $b=(m,n)$. Let $C_1, \ldots, C_r$ be the cells of $\Pc$. To each cell $C_k$ with lower left corner $(i,j)$, we assign a vector $b_k=e_{ij} + e_{i+1,j+1}-e_{i+1,j}-e_{i,j+1} \in \ZZ^{m\times n}$, where $e_{ij}$, $i=1, \ldots, m$, $j=1, \ldots, n$ is the canonical basis of $\ZZ^{m\times n}$. Let $W$ be the sublattice of  $\ZZ^{m \times n}$ spanned by the basis vector $e_{ij}$ with $(i,j) \in V(\Pc)$, then $W \iso \ZZ^s$ where $s=|V(\Pc)|$. As explained before, the vectors $b_1, \ldots, b_r$ form a basis of a saturated lattice $\Lambda$ in $\ZZ^{m \times n}$. Since all $b_i$ belong to $W$, it follows that $\Lambda \subset W$. Therefore, we can complete $b_1, \ldots, b_r$ to a basis $ b_1, \ldots, b_r, b_{r+1}, \ldots, b_s$ of $W$. Let $V$ be a sublattice of $W$ spanned by $b_{r+1}, \ldots, b_s$, and let $\pi : W \rightarrow V$ be the projection map which assigns to $v=\sum_{i=1}^{s} v_i b_i$ the vector $\sum_{i=r+1}^{s} v_i b_i$. Then $\Lambda=\Ker \pi$. Hence $x^a - x^b \in I_{\Lambda}$ if and only if $\pi(a) = \pi(b)$.

Let $\psi' : S \rightarrow K[\{y_k : k=r+1, \ldots , s \}]$ be $K$-algebra homomorphism with $\psi'(x_{ij})=\prod_{k=r+1}^{s} y_k^{v_{ij,k}}$ where the exponents $v_{ij,k}$ are determined by the equation $e_{ij}= \sum_{k=1}^{s}v_{ij,k}b_k$. It follows from the above discussion that $\Ker \psi'= I_{\Lambda}$. Let $a_1=(i_1,j_1), \ldots, a_t=(i_t,j_t)$ be the free vertices of $\Pc$. It is clear that $t=s-r$. We claim that the set $\BB= \{b_1, \ldots, b_r, e_{i_1,j_1}, \ldots , e_{i_t, j_t}\}$ is linearly independent and hence forms a basis of $W$. We order the basis elements $e_{i,j}$ lexicographically. Then we see that the leading term of the $b_k$ is $e_{ij}$, where $(i,j)$ is the lower left corner of $C_k$. Thus we see that all leading terms of the elements of $\BB$ are linearly independent, which implies that elements of $\BB$ are linearly independent. Let $b_{r+k}=e_{i_k,j_k}$. Then the map $\psi'$ coincides with $K$-algebra homomorphism $\psi$ defined in beginning of Section~3. Therefore, $ I_{\Lambda} = \Ker \psi' =\Ker \psi =L_{\Pc}$.
\end{proof}

\begin{Corollary}
\label{prime}
Let $\Pc$ be a collection of cells. Then $I_{\Pc}$ is a prime ideal if and only if $I_{\Pc}=L_{\Pc}$.
\end{Corollary}
\begin{proof}
By Theorem \ref{contained} we have $I_{\Pc} \subset L_{\Pc}$, and by Theorem~\ref{latticeideal} we have $I_{\Lambda}=L_{\Pc}$ where $\Lambda$ is the lattice with basis  $B=b_1,\ldots,b_r$ corresponding to the cells $C_1,\ldots,C_r$ as described in the previous theorem. Hence $I_{\Pc} \subset I_{\Lambda}$. Let $J$ be the lattice basis ideal corresponding to $B$. Then the generators of $J$ are precisely the $2$-adjacent minors in $I_{\Pc}$. In particular, it follows that $J\subset I_{\Pc} \subset I_{\Lambda}$. It is known from \cite[Proposition 1.1]{HS}, that $I_\Lambda= J\: x^\infty$ where $x=\prod_{a\in V(\Pc)}x_a$. Thus if $f\in  I_{\Lambda}$, then there exists an integer $k$ such that $fx^k\in J\subset I_{\Pc}$. Assuming that $I_{\Pc}$ is a prime ideal, it follows that $f\in I_{\Pc}$ since $x\not\in I_{\Pc}$. Hence we see that  $I_{\Pc}=L_{\Pc}$, if $I_{\Pc}$ is a prime ideal.  On the other hand, it is clear that if  $I_{\Pc} = L_{\Pc}$, then  $I_{\Pc}$ is a prime ideal.
\end{proof}
In order to describe the binomials in $I_{\Lambda}$, we introduce some notation. Let $\Pc$ be a collection of cells. We define horizontal and vertical intervals attached to $\Pc$. Let $a=(i,k)$ and $b=(j,k)$ with $i<j$ be in horizontal position. Then $[a,b]$ is called a {\em horizontal interval} of $\Pc$ if $\{(l,k),(l+1,k)\} \in E(\Pc)$ for $l=i, \ldots, j-1$. In addition, if $\{(i-1,k),(i,k)\}$ and $\{(j,k),(j+1,k)\}$ do not belong to $E(\Pc)$, then $[a,b]$ is called maximal horizontal interval of $\Pc$. Similarly we define the vertical intervals attached to $\Pc$. In the Figure~\ref{horizontalinterval}, fat dot marks indicate a maximal horizontal interval $[a,b]$ of $\Pc$.

\begin{figure}[hbt]
\begin{center}
\psset{unit=0.9cm}
\begin{pspicture}(4.5, 1)(4.5,3)
\pspolygon[style=fyp,fillcolor=light](2,2)(2,3)(3,3)(3,2)
\pspolygon[style=fyp,fillcolor=light](3,2)(3,3)(4,3)(4,2)
\pspolygon[style=fyp,fillcolor=light](5,2)(5,3)(6,3)(6,2)
\pspolygon[style=fyp,fillcolor=light](6,2)(6,3)(7,3)(7,2)
\pspolygon[style=fyp,fillcolor=light](2,1)(2,2)(3,2)(3,1)
\pspolygon[style=fyp,fillcolor=light](4,1)(4,2)(5,2)(5,1)
\pspolygon[style=fyp,fillcolor=light](5,1)(5,2)(6,2)(6,1)
\pspolygon[style=fyp,fillcolor=light](3,1)(3,2)(4,2)(4,1)
\pspolygon[style=fyp,fillcolor=light](6,1)(6,2)(7,2)(7,1)
\rput(2,3.3){$a$}
\rput(4,3.3){$b$}
\rput(2,3){$\bullet$}
\rput(3,3){$\bullet$}
\rput(4,3){$\bullet$}
\end{pspicture}
\end{center}
\caption{A maximal horizontal interval}\label{horizontalinterval}
\end{figure}
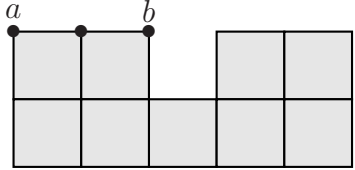

A labeling of $\Pc$ is a function $\alpha : V(\Pc) \rightarrow \ZZ$. The function $\alpha$ is called an {\em admissible labeling} of $\Pc$ if $\alpha([a,b]):=\sum_{c \in [a,b]} \alpha(c) = 0$, for all maximal horizontal and vertical intervals attached to $\Pc$. An example of an admissible labeling of a collection of cells is shown in Figure~\ref{admissible}.

\begin{figure}[hbt]
\begin{center}
\psset{unit=0.9cm}
\begin{pspicture}(4.5,-0.5)(4.5,4)
\rput(-1,0) {
\pspolygon[style=fyp,fillcolor=light](4,0)(4,1)(5,1)(5,0)
\pspolygon[style=fyp,fillcolor=light](3,2)(3,3)(4,3)(4,2)
\pspolygon[style=fyp,fillcolor=light](5,1)(5,2)(6,2)(6,1)
\pspolygon[style=fyp,fillcolor=light](4,2)(4,3)(5,3)(5,2)
\pspolygon[style=fyp,fillcolor=light](4,0)(4,1)(5,1)(5,0)
\pspolygon[style=fyp,fillcolor=light](5,2)(5,3)(6,3)(6,2)
\pspolygon[style=fyp,fillcolor=light](6,0)(6,1)(7,1)(7,0)
\pspolygon[style=fyp,fillcolor=light](6,2)(6,3)(7,3)(7,2)
\pspolygon[style=fyp,fillcolor=light](4,1)(4,2)(5,2)(5,1)
\rput(2.85,3.25){4}
\rput(3.8,3.3){-2}
\rput(4.8,3.3){0}
\rput(5.8,3.3){-3}
\rput(6.9,3.3){1}
\rput(2.8, 1.7){-4}
\rput(3.8,1.7){2}
\rput(4.8,1.7){1}
\rput(5.8,1.7){2}
\rput(6.9, 1.7){-1}
\rput(3.75, 0.7){-1}
\rput(4.8,0.7){0}
\rput(5.8,0.7){3}
\rput(7.25,0.7){-2}
\rput(3.8,-0.3){1}
\rput(4.8,-0.3){-1}
\rput(5.8,-0.3){-2}
\rput(6.9,-0.3){2}
}
\end{pspicture}
\end{center}
\caption{An admissible labeling}\label{marked}
\end{figure}
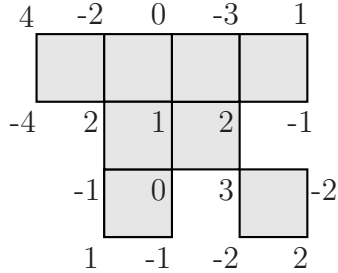

\begin{Lemma}
\label{matrix}
Let $X$ be an $m \times n$ integer matrix with the property that all its column sums are zero. Let $i$ be an integer with
$1 \leq i \leq m$, and suppose that for all $j \neq i$ the row sum for the $j$-th row is zero. Then the row sum of the $i$-th row is also zero.

\end{Lemma}
\begin{proof}
Adding all the rows vectors we obtain a vector $v$ whose components are zero, except possibly at the $i$-th component. Now, because all column sums of $X$ are zero, it follows that sum of the component of $v$ is zero. Hence, the $i$-th component of $v$ must be zero.
\end{proof}

\begin{Theorem}
\label{admissible}
Let $\Pc$ be a collection of cells and $\Lambda$ be the lattice attached to $\Pc$.
\begin{enumerate}
\item[{\em (a)}] If an irreducible binomial belongs to $I_{\Lambda}$, then it is of the form
\[
f_{\alpha} = \prod_{a \in V(\Pc) \atop \alpha(a) > 0} x_a^{\alpha(a)} - \prod_{a \in V(\Pc) \atop \alpha(a) < 0} x_a^{\alpha(a)} ,
 \]
where $\alpha$ is an admissible labeling of $\Pc$.

\item[{\em (b)}] If $\Pc$ is a simple collection of cells and $\alpha$ is an admissible labeling, then $f_{\alpha} \in I_{\Lambda}$.
\end{enumerate}
\end{Theorem}
\begin{proof}
 The lattice $\Lambda \subset \ZZ^{V(\Pc)}$ attached to $\Pc$ consists of all integer vectors $v \in \ZZ^{V(\Pc)}$ which are linear combination of the basis vectors $b_1, \ldots , b_r$ corresponding to the cells of $\Pc$, see the proof of Theorem~\ref{latticeideal}. We claim that if $v = (v_a)_{a \in V(\Pc)} \in \Lambda$, then $\alpha: V(\Pc) \rightarrow \ZZ$, $a \mapsto v_a$, is an admissible labeling, and the converse holds if $\Pc$ is simple.

Let $b_i = (b_{i,a})_{a \in V(\Pc)}$. Then $\alpha_i : V(\Pc) \rightarrow \ZZ$ defined by $a \mapsto b_{i,a}$ is an admissible labeling. Now let $v=\sum_{i=1}^{r} \lambda_i b_i$, $\lambda_i \in \ZZ$. If we let $\alpha: V(\Pc) \rightarrow \ZZ$ be the map $a \mapsto v_a$,  then it follows that $\alpha = \sum_{i=1}^{r} \lambda_i \alpha_i$. Consequently, $\alpha$ is admissible.

Now suppose that $\Pc$ is simple. Let $\alpha: V(\Pc) \rightarrow \ZZ$ be an admissible labeling,
and let $v=(v_a)_{a \in V(\Pc)}$, where $v_a= \alpha(a)$. We want to show that $v \in \Lambda$.
We may assume that $\Pc$ is weakly connected. Let $a=(i,j) \in V(\Pc)$ such that $|a|=i+j$ is minimal.
Then $a$ is a lower left corner of a cell $C$. Let $\lambda = \alpha(a)$.
Then the admissible labeling  $\alpha'=\alpha - \lambda \alpha_i$ has the property that $\alpha' (a)=0$.
We claim that $\alpha'$ is an admissible labeling for $\Pc'= \Pc / \{C\}$. Assuming this,
by induction on the number of cells we obtain the desired conclusion, since $\Pc'$ is again simple.

In order to prove the claim, we first observe the any maximal horizontal or vertical interval of $\Pc$ which has no common vertex with $C$ is also a maximal interval of $\Pc'$, and $\alpha'([a,b])=0$ for any such interval. An interval $[a,b]$ of $\Pc$ is no longer an interval of $\Pc'$ in the cases indicated in the Figure~\ref{notinterval}. The cells marked as dark region, for example $C$ and $E$, represent  cells of $\Pc$.

\begin{figure}[hbt]
\begin{center}
\psset{unit=0.8cm}
\begin{pspicture}(4.5,-2)(4.5,4)
\rput(-4,0){
\pspolygon[style=fyp,fillcolor=light](3,1)(3,2)(4,2)(4,1)
\pspolygon[style=fyp](2,1)(2,2)(3,2)(3,1)
\pspolygon[style=fyp,fillcolor=light](1,1)(1,2)(2,2)(2,1)
\rput(4.5,1.1){$\cdots$}
\rput(4.5,1.9){$\cdots$}
\pspolygon[style=fyp,fillcolor=light](2,0)(2,1)(3,1)(3,0)
\pspolygon[style=fyp,fillcolor=light](5,1)(5,2)(6,2)(6,1)
\rput(0.9,0.75){$a$}
\rput(1.75,0.75){$c$}
\rput(3.3,0.75){$d$}
\rput(6.2,0.75){$b$}
\rput(2.5,0.5){$C$}
\rput(2.5,1.5){$D$}
\rput(1.5,1.5){$E$}
\rput(3.8,-1){A horizontal interval}
}
\rput(15,-2){
\pspolygon[style=fyp,fillcolor=light](-6,1)(-5,1)(-5,2)(-6,2)
\pspolygon[style=fyp](-6,2)(-5,2)(-5,3)(-6,3)
\pspolygon[style=fyp,fillcolor=light](-6,3)(-5,3)(-5,4)(-6,4)
\pspolygon[style=fyp,fillcolor=light](-6,5)(-5,5)(-5,6)(-6,6)
\pspolygon[style=fyp,fillcolor=light](-7,2)(-6,2)(-6,3)(-7,3)
\rput(-6,0.3){A vertical interval}
\rput(-5.5,1.5){$E$}
\rput(-6.5,2.5){$C$}
\rput(-5.5,2.5){$D$}
\rput(-6.2,1){$a$}
\rput(-6.2,1.75){$c$}
\rput(-6.2,3.3){$d$}
\rput(-6.2,6.25){$b$}
\rput(-5.9,4.6){$\vdots$}
\rput(-5.1,4.6){$\vdots$}
}

\end{pspicture}
\end{center}
\caption{}\label{notinterval}
\end{figure}

We discuss only the first case when $[a,b]$ is the horizontal interval. The argument for the case when $[a,b]$ is the vertical interval is similar. As indicated in Figure~\ref{notinterval}, the interval $[a,b]$ splits into two intervals,  namely $[a,c]$ and $[d,b]$. We need to show that $\alpha'([a,c])=0$ and $\alpha'([d,b])=0$. Since $\alpha'$ is an admissible labeling of $\Pc$, we have that $\alpha'([a,b])=0$. Hence, $\alpha'([d,b])=0$ if and only if $\alpha'([a,c])=0$. This is the case if and only if $\alpha'(a)=-\alpha'(c)$. Since $D \notin \Pc$, and $\Pc$ is simple, it follows that $D$ is connected to a border cell of a proper interval in $\NN^2$ whose interior contains $V(\Pc)$. It follows that $\Pc'$ consists of two weakly connected components. We denote by $\mathcal{Q}$ the weakly connected component of $\Pc'$ which contains $E$. Any horizontal or vertical interval $[e,f]$ of $\mathcal{Q}$ different from $[a,c]$ has the property that $\alpha'([e,f])=0$. Let $I$ be the smallest interval of $\NN^2$ containing $\mathcal{Q}$. We extend $\alpha'$ to $\hat{\alpha}: I \rightarrow \ZZ$ by setting $\hat{\alpha} (g)=0$, if $g \notin \mathcal{Q}$. Then, we obtain an integer matrix whose entries are indexed by vertices of $I$ with the property that all column sums and row sums are zero, except the bottom row. By Lemma~\ref{matrix}, it follows that the row sum of the bottom row of this matrix is zero. This implies that $\alpha'(a)=-\alpha'(c)$.

Finally, if $f$ is an irreducible binomial in $I_{\Lambda}$, then $f = x^{v_+} - x^{v_-}$ where $v \in \Lambda$. Therefore, the assertion of the theorem follows from the above discussion.
\end{proof}
Let $\alpha: V(\Pc) \rightarrow \ZZ$ be an admissible labeling of $\Pc$ and $[a,b]$ be an inner interval of $\Pc$ with corners $a$, $b$, $c$ and $d$. Suppose $\alpha(a) \alpha(b) > 0$. Now, we define two admissible labelings of $\Pc$  as follows
\[
\beta(e)= \left\{ \begin{array}{ll}
          1, &  \text{if  $e=a$ or $e=b$}, \\
         -1, &\text{if  $e=c$ or $e=d$}, \\
           0,& \text{elsewhere}.
        \end{array} \right.
\]
and
\[
\alpha'= \left\{ \begin{array}{ll}
        \alpha - \beta, & \text{if  $\alpha(a) > 0$}, \\
        \alpha + \beta, & \text{if  $\alpha(a) < 0$}.
                    \end{array} \right.
\]
We say that $\alpha'$ is obtained from $\alpha$ by a {\em single move}. Similarly, one can define a single move if $\alpha(c) \alpha(d) > 0$ by replacing $a$, $b$ by $c$, $d$ in the definition of $\beta$ and $\alpha'$. We say an admissible labeling $\alpha$ of $\Pc$ reduces to $0$ if there exists a sequence
\[
\alpha=\alpha_0, \alpha_1, \ldots, \alpha_k=0
\]
where each $\alpha_i$ is an admissible labeling of $\Pc$, and  $\alpha_{i+1}$ is obtained from $\alpha_{i}$ by a single move for $i=0, \ldots, k-1$.

\begin{Corollary}
\label{moves}
Let $\Pc$ be a simple collection of cells. Then $I_{\Pc}$ is a prime ideal if and only if each admissible labeling on $\Pc$ reduces to $0$ by a finite number of moves.
\end{Corollary}
\begin{proof}
By Corollary~\ref{prime} we know that $I_{\Pc}$ is a prime ideal if and only if $I_{\Pc}= L_{\Pc}$ and by Theorem~\ref{degree}, this is the case if and only if $L_{\Pc}$ is generated in degree $2$. Since the generators of degree $2$ of $L_{\Pc}$ correspond to simple moves, the assertion follows from Theorem~\ref{admissible}.
\end{proof}

\begin{Theorem}
\label{colconvex}
Let $\Pc$ be a simple collection of cells such that each connected component of $\Pc$ is row or column convex. Then $I_{\Pc}$ is a prime ideal.
\end{Theorem}
\begin{proof}
Let $\alpha$ be an admissible labeling of $\Pc$, then by Corollary~\ref{moves} it is enough to show that $\alpha$ can be reduced to $0$ by a finite number of moves. We may assume that $\Pc$ is weakly connected. By Lemma~\ref{tree}, we know that the graph $G$ attached to $\Pc$ is a tree. Let $\Pc_r$ be a connected component of $\Pc$ such that $r$ is a free vertex of $G$, in other words $r$ is a vertex of order $1$. Let $s$ be the unique element in $V(G)$ such that $\{r,s\} \in E(G)$. Then $|\Pc_r \cap \Pc_s| =1$ and $\Pc_r \cap \Pc_j = \emptyset$, for $j \neq r,s$. We may assume that $\Pc_r$ is column convex with columns $\Cc_1, \ldots, \Cc_n$ such that $\Cc_i \cap \Cc_{i+1} \neq \emptyset$, $i=1, \ldots, n-1$.

We distinguish two cases. In the first case we assume that $n=1$. Then $\Pc_r =[A,B]$ with the lower left corner $a=(i,j)$ of the cell $A$ and $b=(i,k)$ the lower left corner of the cell $B$. We may assume that $\Pc_r \cap \Pc_s = \{(i,j)\}$. Let $c=(i, k+1)$, $d=(i+1, k+1)$ and $e=(i+1,k)$  be the corners of $B$. If $\alpha(c) = 0$, then $\alpha(d)= 0$, because $\alpha$ is an admissible labeling. It follows that $\alpha$ restricted to $\Pc / \{B\}$ is again an admissible labeling. Then by applying induction on the number of cells, we obtain the desired conclusion. Now assume that $\alpha(c) \neq 0$. Without loss of generality we assume that $\alpha(c) > 0$. Then $\alpha(d)=-\alpha(c) < 0$. Since $[f,d]$ is a vertical interval of $\Pc$ where $f=(i+1,j)$, we have $\alpha([f,d])=0$. This shows that  there exists $g \in [f,d]$, $g \neq d$ with $\alpha(g) > 0$. Then there exists a single move $\alpha_1$ obtained from $\alpha$ such that $\alpha_1(c)= \alpha(c) - 1$ and $\alpha_1(g)= \alpha(g) - 1$. Proceeding in this way, we obtain a finite sequence of moves $\alpha=\alpha_0, \alpha_1, \ldots, \alpha_r$, such that $\alpha_r(c) =0$. Then we can remove the cell $B$ as we discussed before.

Now we consider the case when $n \geq 2$. Then either $\Cc_1$ or $\Cc_n$ is disjoint with $\Pc_s$. We may assume that $\Cc_1 \cap \Pc_s= \emptyset$ and that $\Cc_1$ is the left most column of $\Pc_r$. Let $\Cc_1=[A,B]$, and let $a$ be the lower left corner of $A$. We may assume that the maximal horizontal interval containing a is contained in $V(\Pc_r$). Indeed, if this is not the case and $b$ is the upper left corner of $B$, then the  maximal horizontal interval containing $b$ will be contained in $V(\Pc_r)$, and we may replace $a$ by $b$ in the following discussions.

Suppose that $\alpha(a)=0$. Then by  similar arguments as in the case when $n=1$, it follows that $\alpha$ restricted to $\Pc \setminus \{A\}$ is again an admissible labeling. Applying induction on the number of cells, we obtain the desired conclusion. Now we may assume that $\alpha(a) > 0 $. Let $[a,f]$ and $[a,e]$ be the maximal horizontal and vertical intervals of $\Pc$ containing $a$. Since $\alpha([a,e])=0$ and $\alpha([a,f])=0$, there exist $b \in [a,e]$ and $c \in [a,f]$ such that $\alpha(b), \alpha(c) < 0$. Let $[b,g]$ be the maximal horizontal interval of $\Pc$ which contains $b$. Then there exists a vertex $h \in [b,g]$ such that $\alpha(h) > 0$. If $\size ([b,g]) \leq \size([a,f])$, then by using the fact that $\Pc_r$ is column convex, we obtain that $[a,h]$ is an inner interval of $\Pc$, for example, see Figure~\ref{biginterval}. Hence $\alpha(a) \alpha(h) > 0$.

\begin{figure}[hbt]
\begin{center}
\psset{unit=0.9cm}
\begin{pspicture}(4.5,-0.5)(4.5,3)
\rput(-1.5,0){
\pspolygon[style=fyp,fillcolor=light](4,0)(4,1)(5,1)(5,0)
\pspolygon[style=fyp,fillcolor=light](4,1)(4,2)(5,2)(5,1)
\pspolygon[style=fyp,fillcolor=light](4,2)(4,3)(5,3)(5,2)
\pspolygon[style=fyp,fillcolor=light](5,0)(5,1)(6,1)(6,0)
\pspolygon[style=fyp,fillcolor=light](5,1)(5,2)(6,2)(6,1)
\pspolygon[style=fyp,fillcolor=light](6,0)(6,1)(7,1)(7,0)
\pspolygon[style=fyp,fillcolor=light](6,1)(6,2)(7,2)(7,1)
\pspolygon[style=fyp,fillcolor=light](7,0)(7,1)(8,1)(8,0)
\rput(4,0){$\bullet$}
\rput(6,0){$\bullet$}
\rput(8,0){$\bullet$}
\rput(4,2){$\bullet$}
\rput(4,3){$\bullet$}
\rput(5,2){$\bullet$}
\rput(7,2){$\bullet$}

\rput(3.7,-0.3){$a$}
\rput(6,-0.3){$c$}
\rput(8,-0.3){$f$}
\rput(3.7,2){$b$}
\rput(3.7,3){$e$}
\rput(5.2,2.3){$h$}
\rput(7.2,2.3){$g$}
}
\end{pspicture}
\end{center}
\caption{}\label{biginterval}
\end{figure}

If $\size ([b,g]) \geq \size([a,f])$, then by using column convexity of $\Pc_r$, we see that $b$ and $c$ are anti-diagonal corners of an inner interval of $\Pc_r$, for example, see the Figure~\ref{smallinterval}. Hence we have $\alpha(b) \alpha(c) < 0$.

\begin{figure}[hbt]
\begin{center}
\psset{unit=0.9cm}
\begin{pspicture}(4.5,-0.5)(4.5,3)
\rput(-1.5,0){
\pspolygon[style=fyp,fillcolor=light](4,0)(4,1)(5,1)(5,0)
\pspolygon[style=fyp,fillcolor=light](4,1)(4,2)(5,2)(5,1)
\pspolygon[style=fyp,fillcolor=light](4,2)(4,3)(5,3)(5,2)
\pspolygon[style=fyp,fillcolor=light](5,0)(5,1)(6,1)(6,0)
\pspolygon[style=fyp,fillcolor=light](5,1)(5,2)(6,2)(6,1)
\pspolygon[style=fyp,fillcolor=light](6,0)(6,1)(7,1)(7,0)
\pspolygon[style=fyp,fillcolor=light](6,1)(6,2)(7,2)(7,1)
\pspolygon[style=fyp,fillcolor=light](7,1)(7,2)(8,2)(8,1)
\rput(4,0){$\bullet$}
\rput(6,0){$\bullet$}
\rput(7,0){$\bullet$}
\rput(4,2){$\bullet$}
\rput(4,3){$\bullet$}
\rput(8,2){$\bullet$}

\rput(3.7,-0.3){$a$}
\rput(6,-0.3){$c$}
\rput(7,-0.3){$f$}
\rput(3.7,2){$b$}
\rput(3.7,3){$e$}
\rput(8.2,2.3){$g$}
}
\end{pspicture}
\end{center}
\caption{}\label{smallinterval}
\end{figure}

In both cases we obtain $\alpha_1$ from $\alpha$ by a single move such that $\alpha_1 (a) = \alpha(a)-1 $. Then, by repeating the same argument as before we obtain a finite sequence of $\alpha=\alpha_0, \alpha_1, \ldots, \alpha_r$, such that $\alpha(a)$ reduces to 0. It shows that the case when $\alpha(a) > 0$ can be reduced to the case when $\alpha(a)=0$. Hence we obtain the desired result.
\end{proof}

\section{Stack polyominoes}
\label{stackpolyomino}

Let $\Pc = \bigcup_{i=1}^{r} [A_i, B_i]$ be a convex collection of cells, where each $[A_i, B_i]$ is a vertical cell interval. Let $a_i$ be the lower left corner of $A_i$ for $i=1, \ldots, r$. Then $\Pc$ is called {\em stack polyomino} if $[a_1, a_r]$ is a horizontal interval of $\Pc$ with $a_i=a_1+(i-1,j)$ for $i=1,\ldots,r$ and some $j$.  In other words, a collection of cells $\Pc$ is called {\em stack polyomino} if it is a row convex bargraph, see Figure~\ref{stack}. The maximal horizontal interval of $\Pc$ containing $a_1$ is called {\em bottom interval} of $\Pc$, and denoted by $\Bc_{\Pc}$.
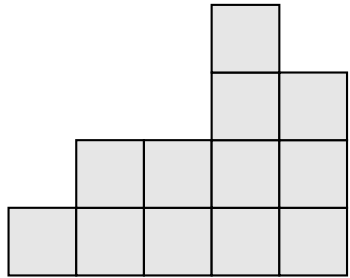
\begin{figure}[hbt]
\begin{center}
\psset{unit=0.9cm}
\begin{pspicture}(4.5,-0.5)(4.5,4)
\rput(-1,0){
\pspolygon[style=fyp,fillcolor=light](3,0)(3,1)(4,1)(4,0)
\pspolygon[style=fyp,fillcolor=light](4,0)(4,1)(5,1)(5,0)
\pspolygon[style=fyp,fillcolor=light](5,0)(5,1)(6,1)(6,0)
\pspolygon[style=fyp,fillcolor=light](6,0)(6,1)(7,1)(7,0)
\pspolygon[style=fyp,fillcolor=light](7,0)(7,1)(8,1)(8,0)
\pspolygon[style=fyp,fillcolor=light](4,1)(4,2)(5,2)(5,1)
\pspolygon[style=fyp,fillcolor=light](5,1)(5,2)(6,2)(6,1)
\pspolygon[style=fyp,fillcolor=light](6,1)(6,2)(7,2)(7,1)
\pspolygon[style=fyp,fillcolor=light](7,1)(7,2)(8,2)(8,1)
\pspolygon[style=fyp,fillcolor=light](6,2)(6,3)(7,3)(7,2)
\pspolygon[style=fyp,fillcolor=light](7,2)(7,3)(8,3)(8,2)
\pspolygon[style=fyp,fillcolor=light](6,3)(6,4)(7,4)(7,3)
}
\end{pspicture}
\end{center}
\caption{A stack polyomino}\label{stack}
\end{figure}

We are going to show that the ideal $I_{\Pc}$ of a stack polyomino $\Pc$ has a quadratic Gr\"obner basis. More generally, let $\Pc$ be an arbitrary collection of cells. We define a total order on the variables $x_a$, $a \in V(\Pc)$ as follows: $x_a > x_b$ with $a=(i,j)$ and $b=(k,l)$, if $i>k$, or $i=k$ and $j>l$. Let $<^1_{\lex}$ be the lexicographical order induced by this order of the variables. Similarly, we denote by $<^2_{\lex}$ the lexicographical order induced by the total order of the variables defined as follows: $x_a > x_b$ with $a=(i,j)$ and $b=(k,l)$, if $i<k$, or $i=k$ and $j>l$. Then we have the following result.

\begin{Theorem}
\label{quadratic}
Let $\Pc$ be a collection of cells. Then the set of inner 2-minors of $\Pc$ form a reduced (quadratic) Gr\"obner basis with respect to $<^1_{\lex}$ if and only if for any two inner intervals $[a,b]$ and $[b,c]$ of $\Pc$, either $[e,c]$ or $[d,c]$ is an inner interval of $\Pc$, where $d$ and $e$ are the anti-diagonal corners of $[a,b]$, see Figure~\ref{quadraticgbasis}.

\begin{figure}[h]
\begin{center}
\psset{unit=0.6cm}
\begin{pspicture}(1,-2)(5,3)
\rput(-1,0.5){
\psline[linestyle=dashed](2,3)(4,3)
\psline[linestyle=dashed](2,1)(2,3)
\psline[linestyle=dashed](4,-1)(6,-1)
\psline[linestyle=dashed](6,-1)(6,1)
\pspolygon[style=fyp,fillcolor=light](2,-1)(4,-1)(4,1)(2,1)
\pspolygon[style=fyp,fillcolor=light](4,1)(6,1)(6,3)(4,3)
\rput(6,3){$\bullet$}
\put(6.2,3){$c$}
\rput(2,1){$\bullet$}
\put(1.6,1.1){$d$}
\rput(4,1){$\bullet$}
\put(3.6,1.1){$b$}
\rput(2,-1){$\bullet$}
\put(1.6,-1.4){$a$}
\rput(4,-1){$\bullet$}
\put(4,-1.4){$e$}
}
\end{pspicture}
\end{center}
\caption{}\label{quadraticgbasis}
\end{figure}

\end{Theorem}
\begin{proof}
Let $\Pc$ be a collection of cells and $\Mc$ be the set of inner 2-minors of $I_{\Pc}$. For any binomial, we always write the leading term as the first term. The set $\Mc$ forms a reduced Gr\"obner basis of $I_{\Pc}$ with respect to $<^1_{\lex}$ if and only if all $S$-polynomials of inner 2-minors of $I_{\Pc}$ reduce to $0$. Take $f_{a,b},f_{r,s} \in \Mc$ given by $f_{a,b} = x_b x_a - x_c x_d$ and $f_{r,s}= x_s x_r - x_p x_q$, where $c$, $d$ are anti-diagonal corners of $[a,b]$, and $p$, $q$ are anti-diagonal corners of $[r,s]$, as shown in Figure~\ref{twocells}.

\begin{figure}[hbt]
\begin{center}
\psset{unit=0.6cm}
\begin{pspicture}(4.5,-1)(4.5,2)
\rput(-3,0){
\pspolygon[style=fyp,fillcolor=light](3,-0.5)(3,1.5)(6,1.5)(6,-0.5)
\rput(2.7,1.5){c}
\rput(2.7,-0.5){a}
\rput(6.3,1.5){b}
\rput(6.3,-0.5){d}
}
\rput(3,0){
\pspolygon[style=fyp,fillcolor=light](3.5,-1)(3.5,2)(5.5,2)(5.5,-1)
\rput(3.4,2.2){p}
\rput(3.4,-1.3){r}
\rput(5.6,2.2){s}
\rput(5.6,-1.3){q}
}
\end{pspicture}
\end{center}
\caption{}\label{twocells}
\end{figure}

We consider the non-trivial case when $\gcd(\ini_{<}(f_{a,b}),\ini_{<}(f_{r,s})) \neq 1$. We may have one of the following possibilities :
(i) $a=r$,
(ii) $b=s$,
(iii) $a=s$ ( or $b=r$).

Consider the case when $a=r$. Without loss of generality, we may assume that $x_b > x_s$. Then $f_{a,s}= x_s x_a - x_p x_q$ and $S(f_{a,b}, f_{a,s}) = x_b x_p x_q - x_s x_c x_d$. Also we may assume that $p \neq c$ and $q \neq d$, otherwise $S(f_{a,b}, f_{a,s})$ reduces to 0 trivially. We have two possible situation, as shown in Figure~\ref{pinsaddle}.

\begin{figure}[hbt]
\begin{center}
\psset{unit=0.6cm}
\begin{pspicture}(5.5,-1.5)(5.5,3.5)
\rput(3,0){
\pspolygon[style=fyp,fillcolor=light](-3,0)(-3,1.2)(-1.5,1.2)(-1.5,0)
\pspolygon[style=fyp,fillcolor=light](-1.5,0)(-1.5,1.2)(0.5,1.2)(0.5,0)
\pspolygon[style=fyp,fillcolor=light](-3,1.2)(-3,3)(0.5,3)(0.5,1.2)
\pspolygon[style=fyp,fillcolor=light](-1.5,1.2)(-1.5,3)(0.5,3)(0.5,1.2)
\rput(-3.3,-0.2){$a$}
\rput(-1.5,-0.3){$q$}
\rput(0.8,-0.2){$d$}
\rput(-3.3,1.4){$p$}
\rput(-1.3,1.4){$s$}
\rput(0.8,1.4){$h$}
\rput(-3.3,3){$c$}
\rput(0.8,3){$b$}
\rput(-1.5, -1.5){$s < b$}
}
\pspolygon[style=fyp,fillcolor=light](9,0)(9,1.2)(10.5,1.2)(10.5,0)
\pspolygon[style=fyp,fillcolor=light](10.5,0)(10.5,1.2)(12.5,1.2)(12.5,0)
\pspolygon[style=fyp,fillcolor=light](9,1.2)(9,3)(10.5,3)(10.5,1.2)
\rput(10.5, -1.5){$s \nless b$}
\rput(12,0){
\rput(-3.3,-0.2){$a$}
\rput(-1.2,-0.3){$d$}
\rput(0.8,-0.2){$q$}
\rput(-3.3,1.3){$p$}
\rput(-1.2,1.6){$h$}
\rput(0.7,1.3){$s$}
\rput(-3.3,3){$c$}
\rput(-1.2,3){$b$}
}
\end{pspicture}
\end{center}
\caption{}\label{pinsaddle}
\end{figure}

When $s<b$, we have
\[
S(f_{a,b}, f_{a,s})= x_q (x_b x_p - x_c x_h ) + x_c (x_h x_q - x_s x_d)
\]

When $s \nleq b$, we have
\[
S(f_{a,b}, f_{a,s})= x_q (x_b x_p - x_c x_h ) - x_c (x_s x_d - x_h x_q)
\]

It shows that in both situations $S(f_{a,b}, f_{a,s})$ reduces to 0 with respect to the inner $2$-minors $f_{p,b}$ and $f_{q,h}$ (or $f_{d,s}$) of $\Pc$, where $h \in [b,d]$ as shown in Figure~\ref{pinsaddle}. Similarly, one shows that $S(f_{a,b}, f_{r,s})$ reduces to 0 when $b=s$.


Now we discuss the only critical case when $a=s$, see Figure~\ref{critical}.

\begin{figure}[hbt]
\begin{center}
\psset{unit=0.8cm}
\begin{pspicture}(4.5,-0.5)(4.5,2.5)
\rput(-2.5,0)
{
\pspolygon[style=fyp,fillcolor=light](5.5,-0.75)(5.5,1)(7,1)(7,-0.75)
\pspolygon[style=fyp,fillcolor=light](7,1)(7,2.25)(9,2.25)(9,1)
\rput(6.8,2.3){c}
\rput(9.2,2.35){b}
\rput(5.3,1){p}
\rput(6.8,1.2){a}
\rput(9.2,1){d}
\rput(5.3,-0.7){r}
\rput(7.2,-0.7){q}
}
\end{pspicture}
\end{center}
\caption{}\label{critical}
\end{figure}

Then $S(f_{a,b}, f_{r,a}) = x_b x_p x_q - x_c x_d x_r$ reduces to 0 if and only if either $[q,b]$ or $[p,b]$ is an inner interval of $\Pc$. This completes the proof.
\end{proof}

\begin{Remark}{\em
\label{order2}
Similarly one can prove the following statement: Let $\Pc$ be a collection of cells. Then the set of inner 2-minors of $\Pc$ form a reduced (quadratic) Gr\"obner basis with respect to $<^2_{\lex}$ if and only if for any two inner intervals $[b,a]$ and $[d,c]$ of $\Pc$ with anti-diagonal corners $e,f$ and $f,g$ as shown in Figure~\ref{quadraticgbasis}, either $b,g$ or $e,c$ are anti-diagonal corners of an inner interval of $\Pc$.}

\begin{figure}[h]
\begin{center}
\psset{unit=0.6cm}
\begin{pspicture}(1,-2)(5,3)
\rput(-1,0.5){
\psline[linestyle=dashed](2,1)(2,-1)
\psline[linestyle=dashed](2,-1)(4,-1)
\psline[linestyle=dashed](6,1)(6,3)
\psline[linestyle=dashed](6,3)(4,3)
\pspolygon[style=fyp,fillcolor=light](4,1)(6,1)(6,-1)(4,-1)
\pspolygon[style=fyp,fillcolor=light](2,1)(2,3)(4,3)(4,1)
\rput(2,3){$\bullet$}
\put(1.5,3.2){$e$}
\rput(4,3){$\bullet$}
\put(4,3.2){$a$}
\rput(2,1){$\bullet$}
\put(1.5,1){$b$}
\rput(4,1){$\bullet$}
\put(3.5,1.3){$f$}
\rput(6,1){$\bullet$}
\put(6.2,1.1){$c$}
\rput(4,-1){$\bullet$}
\put(4,-1.7){$d$}
\rput(6,-1){$\bullet$}
\put(6,-1.5){$g$}
}
\end{pspicture}
\end{center}
\caption{}\label{quadraticgbasis}
\end{figure}

\end{Remark}

\begin{Corollary}
\label{gbasisstcak}
Let $\Pc$ be a stack polyomino. Then the reduced Gr\"obner basis of $I_{\Pc}$ with respect to both monomial orders $<^1_{\lex}$ and $<^2_{\lex}$ consists of all inner 2-minors of $\Pc$.
\end{Corollary}
\begin{proof}
Let $a<b<c$ in $V(\Pc)$ such that $f_{a,b}$ and $f_{b,c}$ are inner 2-minors of $\Pc$. Let $d$ and $e$ be the anti-diagonal corners of the interval $[a,b]$, and $f$ and $g$ be anti-diagonal corners of the interval $[b,c]$. We may assume that $[d,g]$ is a horizontal interval and $[e,f]$ is a vertical interval. It follows from the definition of the stack polyomino that $[e,c]$ is an inner interval of $\Pc$. By applying Theorem~\ref{quadratic}, we obtain that the reduced Gr\"obner basis of $I_{\Pc}$ with respect to $<^1_{\lex}$ consists of all inner 2-minors of $I_{\Pc}$. Similarly one can derive the same conclusion for $<^2_{\lex}$ by applying Theorem~\ref{order2}.
\end{proof}

Let $\Pc$ be a stack polyomino. Now we define a special total order on the variables $x_a$, $a\in V(\Pc)$. Let $[c,d]$ be a vertical interval of maximal size in $\Pc$ and $c=(i,j)$. For any $a,b \in V(\Pc)$ with $a=(k,l)$, $b=(p,q)$, we let $x_a > x_b$ if either (1) $l>q$ , or (2) $l=q$, $k \geq i$, and  $k < p$ or $p < i$, or (3) $l=q$, $k < i$, and $p < k$.

We denote by $<'_{lex}$, the lexicographical term order induced by above order of variables.

\begin{Example}{\em
Let $\Pc$ be the stack polyomino as shown in Figure~\ref{pol}.

\begin{figure}[hbt]
\begin{center}
\psset{unit=0.9cm}
\begin{pspicture}(4.5,-2)(4.5,3)
\rput(-1,0){
\pspolygon[style=fyp,fillcolor=light](3,0)(3,1)(4,1)(4,0)
\pspolygon[style=fyp,fillcolor=light](4,0)(4,1)(5,1)(5,0)
\pspolygon[style=fyp,fillcolor=light](5,0)(5,1)(6,1)(6,0)
\pspolygon[style=fyp,fillcolor=light](6,0)(6,1)(7,1)(7,0)
\pspolygon[style=fyp,fillcolor=light](7,0)(7,1)(8,1)(8,0)
\pspolygon[style=fyp,fillcolor=light](4,1)(4,2)(5,2)(5,1)
\pspolygon[style=fyp,fillcolor=light](5,1)(5,2)(6,2)(6,1)
\pspolygon[style=fyp,fillcolor=light](6,1)(6,2)(7,2)(7,1)
\pspolygon[style=fyp,fillcolor=light](7,1)(7,2)(8,2)(8,1)
\pspolygon[style=fyp,fillcolor=light](6,2)(6,3)(7,3)(7,2)
\pspolygon[style=fyp,fillcolor=light](7,2)(7,3)(8,3)(8,2)
\pspolygon[style=fyp,fillcolor=light](6,3)(6,4)(7,4)(7,3)
\rput(3,1){$\bullet$}
\put(2.7,0.7){$p$}
\rput(4,1){$\bullet$}
\put(3.7,0.7){$q$}
\rput(5,1){$\bullet$}
\put(4.7,0.7){$r$}
\rput(6,1){$\bullet$}
\put(5.7,0.7){$s$}
\rput(7,1){$\bullet$}
\put(6.7,0.7){$t$}
\rput(8,1){$\bullet$}
\put(7.7,0.7){$u$}
\rput(6,0){$\bullet$}
\put(6,-0.3){$c$}
\rput(6,4){$\bullet$}
\put(6,4.3){$d$}
}
\end{pspicture}
\end{center}
\caption{}\label{pol}
\end{figure}

Then for the horizontal interval indicated by fat dot marks, the order of the variables is given as $s > t > u > r > q > p$.
}
\end{Example}

\begin{Remark}{\em
Let $\Pc$ be a stack polyomino and $[c,d]$ be a vertical interval of maximal size in $\Pc$ with $c=(i,j)$. Take $f_{g,h} = x_h x_g - x_p x_q$ be an inner 2-minor of $\Pc$ and $g=(r,s)$. Then we have the following
\begin{enumerate}
\item[(1)] $\ini_< (f_{g,h}) = x_h x_g$ if $r < i$, see Figure~\ref{initial1}.
\begin{figure}[hbt]
\begin{center}
\psset{unit=0.9cm}
\begin{pspicture}(4.5,-0.5)(4.5,3.5)
\rput(-7,0){
\pspolygon[style=fyp,fillcolor=light](4,0)(4,1)(5,1)(5,0)
\pspolygon[style=fyp,fillcolor=light](5,0)(5,1)(6,1)(6,0)
\pspolygon[style=fyp,fillcolor=light](6,0)(6,1)(7,1)(7,0)
\pspolygon[style=fyp,fillcolor=light](7,0)(7,1)(8,1)(8,0)
\pspolygon[style=fyp,fillcolor=light](8,0)(8,1)(9,1)(9,0)
\pspolygon[style=fyp,fillcolor=light](5,1)(5,2)(6,2)(6,1)
\pspolygon[style=fyp,fillcolor=light](6,1)(6,2)(7,2)(7,1)
\pspolygon[style=fyp,fillcolor=light](7,1)(7,2)(8,2)(8,1)
\pspolygon[style=fyp,fillcolor=light](6,2)(6,3)(7,3)(7,2)
\rput(4.8,2.2){$p$}
\rput(4.8,1.2){$g$}
\rput(7.2,2.3){$h$}
\rput(7.15,1.3){$q$}
\rput(6,-0.4){$c$}
\rput(5.9,3.3){$d$}
\rput(6,3){$\bullet$}
\rput(5,1){$\bullet$}
\rput(5,2){$\bullet$}
\rput(6,0){$\bullet$}
\rput(7,1){$\bullet$}
\rput(7,2){$\bullet$}
}
\rput(2,0){
\pspolygon[style=fyp,fillcolor=light](4,0)(4,1)(5,1)(5,0)
\pspolygon[style=fyp,fillcolor=light](5,0)(5,1)(6,1)(6,0)
\pspolygon[style=fyp,fillcolor=light](6,0)(6,1)(7,1)(7,0)
\pspolygon[style=fyp,fillcolor=light](7,0)(7,1)(8,1)(8,0)
\pspolygon[style=fyp,fillcolor=light](8,0)(8,1)(9,1)(9,0)
\pspolygon[style=fyp,fillcolor=light](5,1)(5,2)(6,2)(6,1)
\pspolygon[style=fyp,fillcolor=light](6,1)(6,2)(7,2)(7,1)
\pspolygon[style=fyp,fillcolor=light](7,1)(7,2)(8,2)(8,1)
\pspolygon[style=fyp,fillcolor=light](6,2)(6,3)(7,3)(7,2)
\rput(4.8,2.2){$p$}
\rput(4.8,1.2){$g$}
\rput(6.2,2.3){$h$}
\rput(6.15,1.3){$q$}
\rput(7,-0.4){$c$}
\rput(6.9,3.3){$d$}
\rput(7,3){$\bullet$}
\rput(5,1){$\bullet$}
\rput(5,2){$\bullet$}
\rput(7,0){$\bullet$}
\rput(6,1){$\bullet$}
\rput(6,2){$\bullet$}
}
\end{pspicture}
\end{center}
\caption{}\label{initial1}
\end{figure}
\item[(2)] $\ini_< (f_{g,h}) = x_p x_q$ if $r \geq i$, see Figure~\ref{initial2}.
\begin{figure}[hbt]
\begin{center}
\psset{unit=0.9cm}
\begin{pspicture}(4.5,-0.5)(4.5,1)
\rput(-2,0){
\pspolygon[style=fyp,fillcolor=light](4,0)(4,1)(5,1)(5,0)
\pspolygon[style=fyp,fillcolor=light](5,0)(5,1)(6,1)(6,0)
\pspolygon[style=fyp,fillcolor=light](6,0)(6,1)(7,1)(7,0)
\pspolygon[style=fyp,fillcolor=light](7,0)(7,1)(8,1)(8,0)
\pspolygon[style=fyp,fillcolor=light](8,0)(8,1)(9,1)(9,0)
\pspolygon[style=fyp,fillcolor=light](5,1)(5,2)(6,2)(6,1)
\pspolygon[style=fyp,fillcolor=light](6,1)(6,2)(7,2)(7,1)
\pspolygon[style=fyp,fillcolor=light](7,1)(7,2)(8,2)(8,1)
\pspolygon[style=fyp,fillcolor=light](6,2)(6,3)(7,3)(7,2)
\rput(6.8,2.2){$p$}
\rput(6.8,1.2){$g$}
\rput(8.2,2.3){$h$}
\rput(8.15,1.3){$q$}
\rput(6,-0.4){$c$}
\rput(5.9,3.3){$d$}
\rput(6,3){$\bullet$}
\rput(8,1){$\bullet$}
\rput(8,2){$\bullet$}
\rput(6,0){$\bullet$}
\rput(7,1){$\bullet$}
\rput(7,2){$\bullet$}
}
\end{pspicture}
\end{center}
\caption{}\label{initial2}
\end{figure}
\end{enumerate} }
\end{Remark}

\begin{Theorem}
\label{grobnerstack}
Let $\Pc$ be a stack polyomino and $[c,d]$ be a vertical interval of maximal size in $\Pc$, as shown in Figure~\ref{maxinterval}.

\begin{figure}[hbt]
\begin{center}
\psset{unit=0.9cm}
\begin{pspicture}(4.5,0)(4.5,3.5)
\rput(-1.5,0){
\pspolygon[style=fyp,fillcolor=light](4,0)(4,1)(5,1)(5,0)
\pspolygon[style=fyp,fillcolor=light](5,0)(5,1)(6,1)(6,0)
\pspolygon[style=fyp,fillcolor=light](6,0)(6,1)(7,1)(7,0)
\pspolygon[style=fyp,fillcolor=light](7,0)(7,1)(8,1)(8,0)
\pspolygon[style=fyp,fillcolor=light](5,1)(5,2)(6,2)(6,1)
\pspolygon[style=fyp,fillcolor=light](6,1)(6,2)(7,2)(7,1)
\pspolygon[style=fyp,fillcolor=light](6,2)(6,3)(7,3)(7,2)
\rput(7,-0.3){$c$}
\rput(7,3.3){$d$}
\rput(7,0){$\bullet$}
\rput(7,1){$\bullet$}
\rput(7,2){$\bullet$}
\rput(7,3){$\bullet$}
}
\end{pspicture}
\end{center}
\caption{}\label{maxinterval}
\end{figure}

Then the ideal $(I_{\Pc}, x_c)$ has a squarefree quadratic Gr\"obner basis with respect to $<'_{\lex}$ introduced before.
\end{Theorem}

\begin{proof}
Let $[a,b]$ and $[c,d]$ be the maximal horizontal and vertical interval of $\Pc$ containing $c$ and set $I=(I_{\Pc}, x_c)$. First observe that $I$ is minimally generated by $x_c$ and the set $\Mc$ consisting of the following elements:
\begin{enumerate}
\item[(1)] Those inner 2-minors $f_{g,h}$ of $\Pc$ such that $x_c \notin \supp{f_{g,h}}$
\item[(2)] The degree $2$ monomials $x_e x_f$ with $e \in [a,b]$, $f \in [c,d]$ and $e$ and $f$ are different from $c$ such that either $[e,f]$ is an inner interval of $\Pc$ or $e$ and $f$ are anti-diagonal corners of an inner interval in $\Pc$.
\end{enumerate}

To show that $\Mc \cup \{x_c\}$ is a reduced Gr\"obner basis of $I$ with respect to $<'_{\lex}$, it is enough to show that all $S$-polynomials $S(\mm,\mm')$ , $\mm,\mm' \in \Mc$ reduce to $0$, because the $S$-polynomial $S(\mm, x_c)$, $\mm \in \Mc$ trivially reduces to 0. Take $\mm, \mm' \in \Mc$ and consider the non-trivial case when $\gcd(\ini_{<}(\mm),\ini_{<}(\mm')) \neq 1$.

If $\mm$ and $\mm'$ are both monomials,  then the $S$-polynomial $S(\mm,\mm')$ reduces to $0$ trivially. Next we consider the case when $\mm$ is an inner 2-minor and $\mm'$ is a monomial in $\Mc$. Let $\mm=f_{g,h} = x_h x_g - x_p x_q$ be an inner 2-minor of $\Pc$ and $\mm'=x_ex_f$ with $e \in [a,b]$, $f \in [c,d]$. Let $c=(i,j)$, $h=(k,l)$, $g=(r,s)$, $p=(r,l)$, and $q=(k,s)$. We have following two possibilities:
\begin{enumerate}
\item[(a)] $\ini_< (f_{g,h}) = x_h x_g$, which gives $r < i$
\item[(b)] $\ini_< (f_{g,h} )= x_p x_q$, which gives $r \geq i$
\end{enumerate}

If $\ini_< f_{g,h} = x_h x_g$,then we either have $x_h = x_f$ or $x_g = x_e$. If $x_h=x_f$, then $q \in [c,d]$ and $x_e x_q \in \Mc$, and hence the $S$-polynomial $S(\mm,\mm')=x_e x_p x_q$ reduces to $0$. If $x_g=x_e$, then $q \in [a,b]$ and $x_f x_q \in \Mc$, and the $S$-polynomial $S(\mm,\mm')=x_f x_p x_q$ reduces to $0$.

If $\ini_< (f_{g,h} )= x_p x_q$, then we either have $x_p = x_f$ or $x_q = x_e$. If $x_p=x_f$, then $g \in [c,d]$ and $x_e x_g \in \Mc$, and hence the $S$-polynomial $S(\mm,\mm')=x_e x_h x_g$ reduces to $0$. If $x_q=x_e$, then $g \in [a,b]$ and $x_f x_g \in \Mc$, and the $S$-polynomial $S(\mm,\mm')=x_f x_h x_g$ reduces to $0$.

Now we consider the case when $\mm$ and $\mm'$ are inner 2-minors of $\Pc$. Let $\mm=f_{g,h}=x_h x_g - x_p x_q$ and $\mm'=f_{u,t} = x_t x_u - x_v x_w$. There are three possibilities:

\begin{enumerate}
\item[(a)] $\ini_< (f_{g,h}) = x_h x_g$ and $\ini_< (f_{u,t}) = x_t x_u$
\item[(b)] $\ini_< (f_{g,h}) = x_p x_q$ and $\ini_< (f_{u,t}) = x_v x_w$
\item[(c)] $\ini_< (f_{g,h}) = x_h x_g$ and $\ini_< (f_{u,t}) = x_v x_w$
\end{enumerate}

If $(a)$ holds, then as we have seen in Theorem~\ref{quadratic}, the only non-trivial case to be discusses is when $t=g$, as shown in Figure~\ref{case1}.

\begin{figure}[h]
\begin{center}
\psset{unit=0.6cm}
\begin{pspicture}(1,-2)(5,3)
\rput(-1,0.5){
\psline[linestyle=dashed](4,-1)(6,-1)
\psline[linestyle=dashed](6,-1)(6,1)
\pspolygon[style=fyp,fillcolor=light](2,-1)(4,-1)(4,1)(2,1)
\pspolygon[style=fyp,fillcolor=light](4,1)(6,1)(6,3)(4,3)
\rput(6,3){$\bullet$}
\put(6.2,3){$h$}
\rput(6,1){$\bullet$}
\put(6.2,1.2){$q$}
\rput(4,3){$\bullet$}
\put(3.6,3.1){$p$}
\rput(2,1){$\bullet$}
\put(1.6,1.2){$v$}
\rput(4,1){$\bullet$}
\put(3.6,1.3){$g$}
\rput(2,-1){$\bullet$}
\put(1.6,-1.5){$u$}
\rput(4,-1){$\bullet$}
\put(3.6,-1.5){$w$}
}
\end{pspicture}
\end{center}
\caption{}\label{case1}
\end{figure}

It follows from the definition of stack polyominoes that $f_{w,h} \in I_{\Pc}$. If $x_c \notin \supp (f_{w,h})$ then $f_{w,h} \in \Mc$ and as in Theorem~\ref{quadratic}, we see that $S$-polynomial $S(f_{g,h}, f_{u,g})$ reduces to 0. If $x_c \in \supp (f_{w,h})$, then $h,q \in [c,d]$ and $u,w \in [a,b]$. It shows that the $S$-polynomial $S(f_{g,h}, f_{u,g})= x_h x_v x_w - x_u x_p x_q$ reduces to $0$ in $I$, because $x_h x_w$ and $x_u x_q$ belong to $\Mc$.

If $(b)$ holds, one can argue in a similar way by applying Remark~\ref{order2}.

Now suppose that $(c)$ holds, then for $c=(i,j)$, $g=(r,s)$ and $u=(m,n)$, we get $r < i \leq m$. In this case we either have $w=h$ or $v=h$.

Let $w=h$, then $S(f_{g,h}, f_{u,t}) = x_v x_p x_q - x_t x_u x_g$. By definition of stack, there exists $z \in V(\Pc)$ such that $f_{z,t} = x_t x_z - x_v x_q \in I_{\Pc}$. If $x_c \notin \supp (f_{z,t})$, then $f_{z,t} \in \Mc$ and $S(f_{g,h}, f_{u,t}) = x_p (x_v x_q - x_t x_z) - x_t (x_u x_g - x_p x_z)$ reduces to 0, see Figure~\ref{exp}.

If $x_c \in \supp (f_{z,t})$, then $x_v x_q$, $x_u x_g \in \Mc$ and again $S(f_{g,h},f_{u,t})$ reduces to 0.

\begin{figure}[hbt]
\begin{center}
\psset{unit=0.9cm}
\begin{pspicture}(4.5,-0.5)(4.5,3.5)
\rput(-2,0){
\pspolygon[style=fyp,fillcolor=light](4,0)(4,1)(5,1)(5,0)
\pspolygon[style=fyp,fillcolor=light](5,0)(5,1)(6,1)(6,0)
\pspolygon[style=fyp,fillcolor=light](6,0)(6,1)(7,1)(7,0)
\pspolygon[style=fyp,fillcolor=light](7,0)(7,1)(8,1)(8,0)
\pspolygon[style=fyp,fillcolor=light](8,0)(8,1)(9,1)(9,0)
\pspolygon[style=fyp,fillcolor=light](5,1)(5,2)(6,2)(6,1)
\pspolygon[style=fyp,fillcolor=light](6,1)(6,2)(7,2)(7,1)
\pspolygon[style=fyp,fillcolor=light](7,1)(7,2)(8,2)(8,1)
\pspolygon[style=fyp,fillcolor=light](6,2)(6,3)(7,3)(7,2)
\rput(6.8,2.2){$v$}
\rput(6.8,1.2){$u$}
\rput(6.8,-0.3){$z$}
\rput(4.8,1.2){$p$}
\rput(4.8,-0.3){$g$}
\rput(8.2,2.3){$t$}
\rput(8.65,1.3){$w=h$}
\rput(8.2,-0.3){$q$}
\rput(6,-0.3){$c$}
\rput(5.9,3.3){$d$}
\rput(6,3){$\bullet$}
\rput(8,1){$\bullet$}
\rput(8,2){$\bullet$}
\rput(6,0){$\bullet$}
\rput(8,0){$\bullet$}
\rput(7,0){$\bullet$}
\rput(5,0){$\bullet$}
\rput(5,1){$\bullet$}
\rput(7,1){$\bullet$}
\rput(7,2){$\bullet$}
}
\end{pspicture}
\end{center}
\caption{}\label{exp}
\end{figure}

Now, let $v=h$. Then $S(f_{g,h}, f_{u,t}) = x_g x_t x_u - x_w x_p x_q$. Again, by definition of stack, there exists $l \in V(\Pc)$ such that $f_{g,t} = x_t x_g - x_l x_p \in I_{\Pc}$, see Figure~\ref{h=v}.
\begin{figure}[hbt]
\begin{center}
\psset{unit=0.9cm}
\begin{pspicture}(4.5,-0.5)(4.5,4.5)
\rput(-1,0){
\pspolygon[style=fyp,fillcolor=light](3,0)(3,1)(4,1)(4,0)
\pspolygon[style=fyp,fillcolor=light](4,0)(4,1)(5,1)(5,0)
\pspolygon[style=fyp,fillcolor=light](5,0)(5,1)(6,1)(6,0)
\pspolygon[style=fyp,fillcolor=light](6,0)(6,1)(7,1)(7,0)
\pspolygon[style=fyp,fillcolor=light](7,0)(7,1)(8,1)(8,0)
\pspolygon[style=fyp,fillcolor=light](4,1)(4,2)(5,2)(5,1)
\pspolygon[style=fyp,fillcolor=light](5,1)(5,2)(6,2)(6,1)
\pspolygon[style=fyp,fillcolor=light](6,1)(6,2)(7,2)(7,1)
\pspolygon[style=fyp,fillcolor=light](7,1)(7,2)(8,2)(8,1)
\pspolygon[style=fyp,fillcolor=light](6,2)(6,3)(7,3)(7,2)
\pspolygon[style=fyp,fillcolor=light](5,2)(5,3)(6,3)(6,2)
\pspolygon[style=fyp,fillcolor=light](7,2)(7,3)(8,3)(8,2)
\pspolygon[style=fyp,fillcolor=light](6,3)(6,4)(7,4)(7,3)
\rput(6.8,3.2){$h$}
\rput(6.8,2.2){$u$}
\rput(6.8,1.2){$q$}
\rput(4.8,1.2){$g$}
\rput(4.8,3.1){$p$}
\rput(8.2,2.3){$w$}
\rput(8.2,1.2){$l$}
\rput(8.2,3.3){$t$}
\rput(6,-0.3){$c$}
\rput(5.9,4.3){$d$}
\rput(7,3){$\bullet$}
\rput(5,3){$\bullet$}
\rput(6,4){$\bullet$}
\rput(8,3){$\bullet$}
\rput(8,2){$\bullet$}
\rput(6,0){$\bullet$}
\rput(8,1){$\bullet$}
\rput(7,0){$\bullet$}
\rput(5,1){$\bullet$}
\rput(7,1){$\bullet$}
\rput(7,2){$\bullet$}
}
\end{pspicture}
\end{center}
\caption{}\label{h=v}
\end{figure}
It is clear that $x_c \notin \supp (f_{g,t})$, and $S(f_{g,h}, f_{u,t}) = x_u (x_g x_t - x_l x_p) + x_p (x_u x_l - x_w x_q)$ reduces to 0.
\end{proof}

\begin{Corollary}
With the notation introduced in Theorem~\ref{grobnerstack}, we have that $(I_P, x_c)$ is a radical ideal.
\end{Corollary}
\begin{proof}
It is a known fact, see for example \cite[Proof of Cor. 2.2]{HHH}, that an ideal is reduced if it has a squarefree initial ideal. Hence the assertion follows from Theorem~\ref{grobnerstack}.
\end{proof}

Next, we are going to determine the minimal prime ideals of $(I_{\Pc}, x_c)$ where $c$ is chosen as in Theorem~\ref{grobnerstack}. To this end we introduce some notation. For each element $a \in V(\Pc)$ there exists a unique element $\pi(a) \in \Bc_{\Pc}$ such that $a$ and $\pi(a)$ are in vertical position, see Figure~\ref{pi(a)}.

\begin{figure}[hbt]
\begin{center}
\psset{unit=0.9cm}
\begin{pspicture}(4.5,-0.5)(4.5,3)
\rput(-1.5,0){
\pspolygon[style=fyp,fillcolor=light](4,0)(4,1)(5,1)(5,0)
\pspolygon[style=fyp,fillcolor=light](5,0)(5,1)(6,1)(6,0)
\pspolygon[style=fyp,fillcolor=light](6,0)(6,1)(7,1)(7,0)
\pspolygon[style=fyp,fillcolor=light](7,0)(7,1)(8,1)(8,0)
\pspolygon[style=fyp,fillcolor=light](5,1)(5,2)(6,2)(6,1)
\pspolygon[style=fyp,fillcolor=light](6,1)(6,2)(7,2)(7,1)
\pspolygon[style=fyp,fillcolor=light](6,2)(6,3)(7,3)(7,2)
\rput(6,-0.4){$\pi(a)$}
\rput(5.8,2.3){$a$}
\rput(6,0){$\bullet$}
\rput(6,2){$\bullet$}
}
\end{pspicture}
\end{center}
\caption{}\label{pi(a)}
\end{figure}

\begin{Theorem}
\label{notsodifficult}
With the notation and assumptions introduced in Theorem~\ref{grobnerstack}, let $P$ be a prime ideal containing $(I_{\Pc}, x_c)$, and let $[c,e]$ be the maximal subinterval of $[c,d]$ with the property that $x_f \in P$ for all $f \in [c,e]$. Assume that $e \neq d$. Let $[g,h]$ be the smallest horizontal interval of $\Pc$ with $e \in [g,h]$ and $g,h \in \partial \Pc$. Then $Q_e=(I_{\Pc}, \{x_p \: p \in [\pi(g), h]\})$ is a prime ideal with $(I_{\Pc}, x_c) \subset Q_e \subset P$.
\end{Theorem}
 Figure~\ref{minimalprime} displays the situation as described in Theorem~\ref{notsodifficult}.

\begin{figure}[hbt]
\begin{center}
\psset{unit=0.9cm}
\begin{pspicture}(4.5,-0.5)(4.5,3.5)
\rput(-2,0){
\pspolygon[style=fyp,fillcolor=light](4,0)(4,1)(5,1)(5,0)
\pspolygon[style=fyp,fillcolor=light](5,0)(5,1)(6,1)(6,0)
\pspolygon[style=fyp,fillcolor=light](6,0)(6,1)(7,1)(7,0)
\pspolygon[style=fyp,fillcolor=light](7,0)(7,1)(8,1)(8,0)
\pspolygon[style=fyp,fillcolor=light](5,1)(5,2)(6,2)(6,1)
\pspolygon[style=fyp,fillcolor=light](6,1)(6,2)(7,2)(7,1)
\pspolygon[style=fyp,fillcolor=light](6,2)(6,3)(7,3)(7,2)
\pspolygon[style=fyp,fillcolor=light](7,1)(7,2)(8,2)(8,1)
\pspolygon[style=fyp,fillcolor=light](7,2)(7,3)(8,3)(8,2)
\pspolygon[style=fyp,fillcolor=light](8,0)(8,1)(9,1)(9,0)
\rput(6,-0.4){$\pi(g)$}
\rput(5.8,2.3){$g$}
\rput(8.2,2.3){$h$}
\rput(8.2,-0.4){$\pi(h)$}
\rput(7.2,2.3){$e$}
\rput(7,-0.4){$c$}
\rput(6.9,3.3){$d$}
\rput(6,0){$\bullet$}
\rput(6,1){$\bullet$}
\rput(6,2){$\bullet$}
\rput(7,0){$\bullet$}
\rput(7,1){$\bullet$}
\rput(7,2){$\bullet$}
\rput(8,0){$\bullet$}
\rput(8,1){$\bullet$}
\rput(8,2){$\bullet$}
}
\end{pspicture}
\end{center}
\caption{}\label{minimalprime}
\end{figure}

\begin{proof}[Proof of Theorem~\ref{notsodifficult}]
We first show the inclusions $(I_{\Pc}, x_c) \subset Q_e \subset P$. Obviously, $(I_{\Pc}, x_c) \subset Q$. Let $p \in [\pi(g), h]$. If $p \in [c,e]$, then $x_p \in P$. Assume now that $ p \notin [c,e]$. Let $e=(i,j)$, and define $e'=(i,j+1)$. Then $x_{e'} \notin P$. We consider the smallest interval $\Ic$ of $V(\Pc)$ containing $e'$ and $p$. Let $e'$, $p$, $q$ and $r$ be the corners of $\Ic$. Then $\Ic$ is an inner interval of $\Pc$ and either $q$ or $r$ belongs to the interval $[c,e]$. Say, $q \in [c,e]$. Then $x_{e'} x_p - x_q x_r \in I_{\Pc} \subset P$ and $x_q \in \Pc$. Therefore, $x_{e'}x_p \in \Pc$. Since $x_{e'} \notin P$ and $P$ is a prime ideal, it follows that $x_p \in P$. It shows $Q_e \subset \Pc$.

It remains to be shown that $Q_e$ is a prime ideal. Observe that $Q_e= (J, \{x_p \: p \in [\pi(g), h]\})$ where $J$ is generated by the minors of the form
\begin{enumerate}
\item[(i)] $f_{a,b} \in I_{\Pc}$ with $a=(k,l)$ and $l>j$.
\item[(ii)] $f_{a,b} \in I_{\Pc}$ with $a=(k,l)$ and $l<j$ and $a,b \notin [\pi(g),h]$
\end{enumerate}

The ideal $J_1$ generated by the minors in (i) is the ideal of inner 2-minors of a stack polyomino $\Pc'$ which consists of cells of $\Pc$ with lower left corner $a=(k,l)$ with $l > j$. Hence $J_1$ is a prime ideal.

Let $V=\{a \in V(\Pc) \: a=(r,s), s \leq j \; \text{and} \; a \notin [\pi(g),h]\}$, and $\pi(g)=(i_1,t)$, $\pi(h)=(i_2,t)$. We define a  map $\alpha\: V \rightarrow V$ given by
\[
\alpha(a)= \left\{ \begin{array}{ll}
          a, &  \text{if  $r < i_1$}, \\
         (r-(i_2-i_1+1),s), &\text{if  $r > i_2$}.
                  \end{array} \right.
\]
With the new co-ordinate assigned by $\alpha$, the ideal $J_2$ generated by the inner  $2$-minors in (ii) may again be identified by the ideal of inner 2-minors of a stack polyomino whose vertex set is contained in $\alpha(V)$. Furthermore, the generator of $J_1$ and $J_2$ have disjoint support. This implies that $(J_1, J_2)$ is a prime ideal. Since $J=(J_1,J_2)$, we conclude that $Q_e$ is a prime ideal.
\end{proof}

A vertex $a \in \partial \Pc$ is called an {\em inside (outside) corner}  of the stack polyomino $\Pc$ if it belongs to exactly three (one) different cells of $\Pc$, see Figure~\ref{insideoutside} in which inside and outside corners are shown by fat dots.

\begin{figure}[hbt]
\begin{center}
\psset{unit=0.9cm}
\begin{pspicture}(4.5,-0.5)(4.5,3)
\rput(-5,0){
\pspolygon[style=fyp,fillcolor=light](4,0)(4,1)(5,1)(5,0)
\pspolygon[style=fyp,fillcolor=light](5,0)(5,1)(6,1)(6,0)
\pspolygon[style=fyp,fillcolor=light](6,0)(6,1)(7,1)(7,0)
\pspolygon[style=fyp,fillcolor=light](7,0)(7,1)(8,1)(8,0)
\pspolygon[style=fyp,fillcolor=light](5,1)(5,2)(6,2)(6,1)
\pspolygon[style=fyp,fillcolor=light](6,1)(6,2)(7,2)(7,1)
\pspolygon[style=fyp,fillcolor=light](6,2)(6,3)(7,3)(7,2)
\rput(5,1){$\bullet$}
\rput(6,2){$\bullet$}
\rput(7,1){$\bullet$}
\rput(6.2,-0.5){Inside corners}
}
\rput(2.5,0){
\pspolygon[style=fyp,fillcolor=light](4,0)(4,1)(5,1)(5,0)
\pspolygon[style=fyp,fillcolor=light](5,0)(5,1)(6,1)(6,0)
\pspolygon[style=fyp,fillcolor=light](6,0)(6,1)(7,1)(7,0)
\pspolygon[style=fyp,fillcolor=light](7,0)(7,1)(8,1)(8,0)
\pspolygon[style=fyp,fillcolor=light](5,1)(5,2)(6,2)(6,1)
\pspolygon[style=fyp,fillcolor=light](6,1)(6,2)(7,2)(7,1)
\pspolygon[style=fyp,fillcolor=light](6,2)(6,3)(7,3)(7,2)
\rput(4,0){$\bullet$}
\rput(4,1){$\bullet$}
\rput(5,2){$\bullet$}
\rput(6,3){$\bullet$}
\rput(7,3){$\bullet$}
\rput(8,0){$\bullet$}
\rput(8,1){$\bullet$}
\rput(6.2, -0.5){Outside corners}
}
\end{pspicture}
\end{center}
\caption{}\label{insideoutside}
\end{figure}

In the situation of Theorem~\ref{notsodifficult}, we define the following prime ideals.
\begin{enumerate}
\item[(1)] $P_1 =( I_{\Pc}, \{x_l \: l \in \Bc_{\Pc} \})$, $P_2=(I_{\Pc}, \{x_l \: l \in [c,d]\})$
\item[(2)] Let $e_1, \ldots, e_s$ be the elements of $[c,d]$ with the property that the maximal horizontal interval of $\Pc$ which contains $e_i$ also contains an inside corner of $\Pc$. For simplicity, we set $Q_i= Q_{e_i}$, where $Q_{e_i}$ is defined as in Theorem~\ref{notsodifficult}.
\end{enumerate}

\begin{Corollary}
\label{minprime}
The minimal prime ideals of $(I_{\Pc}, x_c)$ are $P_1, P_2$ and  $Q_1, Q_2, \ldots, Q_s$.
\end{Corollary}
\begin{proof}
Since $P_1 = Q_c$, it follows that $P_1$ is a prime ideal. Observe that $P_2= (J, \{x_l \: l \in [c,d]\})$ where $J$ is the ideal generated by inner 2-minors $f_{a,b} \in I_{\Pc}$ with $a,b \notin [c,d]$. With the same argument as in the proof of Theorem~\ref{notsodifficult},
it follows that $P_2$ is a prime ideal.

Suppose $P_1$ is not a minimal prime ideal of $(I_{\Pc}, x_c)$ and let $P$ be a prime ideal such that $(I_{\Pc},x_c) \subset P \subsetneq P_1$. Then there exists a vertex $a \in \Bc_{\Pc}$ such that $x_a \notin P$, and an inner 2-minor $x_a x_h - x_c x_g$ of $\Pc$ where $g, h \notin \Bc_{\Pc}$. Since $x_c \in P$ and $x_a \notin P$, it follows that $x_h \in P$ and hence $x_h \in P_1$, a contradiction. Similarly one shows that $P_2$ is a minimal prime ideal of $(I_{\Pc}, x_c)$.

Let $Q$ be a minimal prime ideal of $(I_{\Pc}, x_c)$ different from $P_1$ and $P_2$. Then there exists $k \in [c,d]$ such that $x_k \notin J$. We let $[c,e]$ be the maximal subinterval of $[c,d]$ with the property that $x_f \in Q$, for all $f \in [c,e]$. Since $e \neq d$, by applying Theorem~\ref{notsodifficult} we see that $Q_e \subset Q$. Minimality of $Q$ implies that $Q_e = Q$. Now we show that $Q=Q_i$, for some $1 \leq i \leq s$.

Let $Q_e=(I_{\Pc}, \{x_p \: p \in [\pi(g), h]\})$ as described in Theorem~\ref{notsodifficult}, and suppose that neither  $g$ nor $h$ is an inside corner of $\Pc$. If $[\pi(g),g]$ and $[\pi(h),h]$ do not contain any inside corner of $\Pc$, then $P_1 \subset Q_e$, and $Q_e$ is not a minimal prime ideal of $(I_{\Pc}, x_c)$. Otherwise, we may assume that there exists an inside corner in either $[\pi(g),g]$ or $[\pi(h),h]$. If both intervals contain an inside corner then we let $p$ be the inside corner with greater y-coordinate. We may assume that $p \in [\pi(g),g]$. Observe that $p$ is uniquely determined. Moreover, there exist two uniquely determined vertices $f \in [c,d]$ and  $q \in [\pi(h),h]$ such that $f$, $p$ and $q$ are in horizontal position. Since $\pi(g) = \pi(p)$ and $\pi(h) = \pi(q)$ which implies that $[\pi(p), q] \subset [\pi(g), h] $, and that $Q_f \subsetneq Q_e$. Hence $Q_e$ is not a minimal prime ideal if none of $g$ and $h$ is an inside corner of $\Pc$.

Let $Q_i=(I_{\Pc}, \{x_p \: p \in [\pi(g), h]\})$ as described in (2) such that either $g$ or $h$ is an inside corner of $\Pc$. We may assume that $g$ is an inside corner of $\Pc$ and $g=(k,l)$. Assume that $Q_i$ is not a minimal prime ideal and $P$ be a prime ideal such that $(I_{\Pc},x_c) \subset P \subsetneq Q_i$. Then there exists a vertex $r \in [\pi(g),h]$ such that $x_r \notin P$. Let $\Ic_1$ and $\Ic_2$ be the vertical and horizontal intervals respectively such that $r \in \Ic_1$, $\Ic_2$ and $\Ic_1, \Ic_2 \subset [\pi(g),h]$, and $\Ic_1$ and $\Ic_2$ are maximal with this property. Since $g$ is an inside corner, the vertices $g'=(k-1, l)$ and $g''=(k,l+1)$ belong to $\partial \Pc$. We see that $r$ has the property that for any vertex $s \in \Ic_1$, $s \neq r$,there exists an inner 2-minor $x_u x_s - x_r x_t$ of $\Pc$ where $t,u \in [\pi(g'), g']$. Since $P$ is a prime ideal and $x_r ,x_t \notin P$, we obtain that $x_s \notin P$, for all $s \in \Ic_1$. In particular $c \notin \Ic_1$, because $x_c \in P$. Also for any $s' \in \Ic_2$, there exists an inner 2-minor $x_r x_{t'} - x_{s'} x_{u'}$ of $\Pc$ such that $u', t'$ and $g''$ are in horizontal position. Again, because $P$ is a prime ideal and $x_r, x_{t'} \notin P$, we conclude that $x_s' \notin P$ for all $s' \in \Ic_2$. On the other hand, since $c \notin \Ic_1$ there exist $s \in I_1$ and $s' \in \Ic_2$ such that $x_r x_c - x_s x_{s'}$ is an inner 2-minor of $\Pc$. By using $x_c \in P$, it follows that either $x_s$ or $x_{s'}$ belongs to $P$, a contradiction. Hence we conclude that $Q_i$ is a minimal prime ideal of $(I_{\Pc}, x_c)$.
\end{proof}

In Figure~\ref{minimalp} we display a stack polyomino $\Pc$ and all the minimal prime ideals of $(I_{\Pc}, x_c)$ as described above. The fat dots mark the interval attached to the minimal prime ideals and the dark shadowed areas, the region where the inner $2$-minor have to be taken.

\begin{figure}[hbt]
\begin{center}
\psset{unit=0.7cm}
\begin{pspicture}(4.5,-7)(4.5,4)

\rput(2,0){
\pspolygon[style=fyp,fillcolor=light](-5,0)(-5,1)(-4,1)(-4,0)
\pspolygon[style=fyp,fillcolor=light](-4,0)(-4,1)(-3,1)(-3,0)
\pspolygon[style=fyp,fillcolor=light](-3,0)(-3,1)(-2,1)(-2,0)
\pspolygon[style=fyp,fillcolor=light](-2,0)(-2,1)(-1,1)(-1,0)
\pspolygon[style=fyp,fillcolor=light](-3,1)(-3,2)(-2,2)(-2,1)
\pspolygon[style=fyp,fillcolor=light](-4,1)(-4,2)(-3,2)(-3,1)
\pspolygon[style=fyp,fillcolor=light](-3,2)(-3,3)(-2,3)(-2,2)
\pspolygon[style=fyp,fillcolor=light](-4,2)(-4,3)(-3,3)(-3,2)
\pspolygon[style=fyp,fillcolor=light](-4,3)(-4,4)(-3,4)(-3,3)
\rput(-3,0){$\bullet$}
\rput(-2.8,-0.2){$c$}
\rput(-2.5,-0.7){$\Pc$}
}

\rput(8,0){
\pspolygon[style=fyp](-5,0)(-5,1)(-4,1)(-4,0)
\pspolygon[style=fyp](-4,0)(-4,1)(-3,1)(-3,0)
\pspolygon[style=fyp](-3,0)(-3,1)(-2,1)(-2,0)
\pspolygon[style=fyp](-2,0)(-2,1)(-1,1)(-1,0)
\pspolygon[style=fyp,fillcolor=light](-3,1)(-3,2)(-2,2)(-2,1)
\pspolygon[style=fyp,fillcolor=light](-4,1)(-4,2)(-3,2)(-3,1)
\pspolygon[style=fyp,fillcolor=light](-3,2)(-3,3)(-2,3)(-2,2)
\pspolygon[style=fyp,fillcolor=light](-4,2)(-4,3)(-3,3)(-3,2)
\pspolygon[style=fyp,fillcolor=light](-4,3)(-4,4)(-3,4)(-3,3)
\rput(-5,0){$\bullet$}
\rput(-4,0){$\bullet$}
\rput(-3,0){$\bullet$}
\rput(-2,0){$\bullet$}
\rput(-1,0){$\bullet$}
\rput(-2.9,-0.7){$P_1$}
}

\rput(14,0){
\pspolygon[style=fyp,fillcolor=light](-5,0)(-5,1)(-4,1)(-4,0)
\pspolygon[style=fyp](-4,0)(-4,1)(-3,1)(-3,0)
\pspolygon[style=fyp](-3,0)(-3,1)(-2,1)(-2,0)
\pspolygon[style=fyp,fillcolor=light](-2,0)(-2,1)(-1,1)(-1,0)
\pspolygon[style=fyp](-3,1)(-3,2)(-2,2)(-2,1)
\pspolygon[style=fyp](-4,1)(-4,2)(-3,2)(-3,1)
\pspolygon[style=fyp](-3,2)(-3,3)(-2,3)(-2,2)
\pspolygon[style=fyp](-4,2)(-4,3)(-3,3)(-3,2)
\pspolygon[style=fyp](-4,3)(-4,4)(-3,4)(-3,3)
\rput(-3,0){$\bullet$}
\rput(-3,1){$\bullet$}
\rput(-3,2){$\bullet$}
\rput(-3,3){$\bullet$}
\rput(-3,4){$\bullet$}
\rput(-2.9,-0.7){$P_2$}
}

\rput(2,-6){
\pspolygon[style=fyp](-5,0)(-5,1)(-4,1)(-4,0)
\pspolygon[style=fyp](-4,0)(-4,1)(-3,1)(-3,0)
\pspolygon[style=fyp](-3,0)(-3,1)(-2,1)(-2,0)
\pspolygon[style=fyp](-2,0)(-2,1)(-1,1)(-1,0)
\pspolygon[style=fyp](-3,1)(-3,2)(-2,2)(-2,1)
\pspolygon[style=fyp](-4,1)(-4,2)(-3,2)(-3,1)
\pspolygon[style=fyp,fillcolor=light](-3,2)(-3,3)(-2,3)(-2,2)
\pspolygon[style=fyp,fillcolor=light](-4,2)(-4,3)(-3,3)(-3,2)
\pspolygon[style=fyp,fillcolor=light](-4,3)(-4,4)(-3,4)(-3,3)
\rput(-4,0){$\bullet$}
\rput(-3,0){$\bullet$}
\rput(-2,0){$\bullet$}
\rput(-4,1){$\bullet$}
\rput(-3,1){$\bullet$}
\rput(-2,1){$\bullet$}
\rput(-2.9,-0.7){$Q_1$}
}

\rput(8,-6){
\pspolygon[style=fyp](-5,0)(-5,1)(-4,1)(-4,0)
\pspolygon[style=fyp](-4,0)(-4,1)(-3,1)(-3,0)
\pspolygon[style=fyp](-3,0)(-3,1)(-2,1)(-2,0)
\pspolygon[style=fyp,fillcolor=light](-2,0)(-2,1)(-1,1)(-1,0)
\pspolygon[style=fyp](-3,1)(-3,2)(-2,2)(-2,1)
\pspolygon[style=fyp](-4,1)(-4,2)(-3,2)(-3,1)
\pspolygon[style=fyp](-3,2)(-3,3)(-2,3)(-2,2)
\pspolygon[style=fyp](-4,2)(-4,3)(-3,3)(-3,2)
\pspolygon[style=fyp](-4,3)(-4,4)(-3,4)(-3,3)
\rput(-4,0){$\bullet$}
\rput(-3,0){$\bullet$}
\rput(-4,1){$\bullet$}
\rput(-3,1){$\bullet$}
\rput(-4,2){$\bullet$}
\rput(-3,2){$\bullet$}
\rput(-4,3){$\bullet$}
\rput(-3,3){$\bullet$}
\rput(-2.9,-0.7){$Q_2$}
}

\end{pspicture}
\end{center}
\caption{The minimal prime ideals of a stack polyomino}\label{minimalp}
\end{figure}
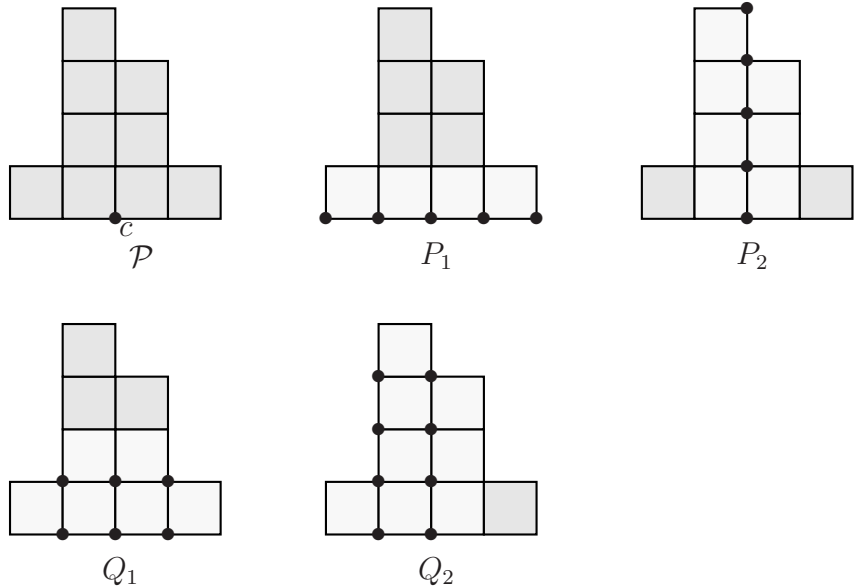

Let $q_i= Q_i / I_{\Pc}$ and $p_j= P_j / I_{\Pc}$, for $i=1,\ldots,s$ and $j=1,2$ be the residue classes $Q_i$ and $P_j$ in $K[\Pc]$.

\begin{Corollary}
\label{classgroup}
The class group $\Cl (K[\Pc])$ of $K[\Pc]$ is free of rank $s+1$ with basis
 \[
\cl(q_1), \ldots, \cl(q_s), \cl(p_1).
\]
\end{Corollary}
\begin{proof}
Let $[c,d]$ be a vertical interval of $\Pc$ of maximal size and $[a,b] = \Bc_{\Pc}$. Then for any $p \in \Pc \setminus ([a,b] \cup [c,d] )$, there exists an inner 2-minor $x_c x_p - x_r x_s$ of $\Pc$ where $r \in [a,b]$ and $s \in [c,d]$. Thus, in $K[\Pc]_{x_s}$ we have $x_p = x_r x_s x_c^{-1}$. It follows that $K[\Pc]_{x_c} = K [\{x_l \: l \in [a,b] \cup [c,d]]_{x_c}$. From Theorem~\ref{convexdimension}, we know that $\dim K[\Pc]_{x_c} =\dim K[\Pc] = |[a,b]| + |[c,d]| - 1$. Hence $K [\{x_l \: l \in [a,b] \cup [c,d]\}]$ is a polynomial ring. Consequently, $K[\Pc]_{x_c}$ is factorial. By applying Nagata's Lemma \cite[Corollary 7.2]{Nag}, Corollary~\ref{minprime} implies that $\Cl(K[\Pc])$ is generated by $ \cl(q_1), \ldots, \cl(q_s), \cl(p_1), \cl(p_2)$. Since $(x_c) = \bigcap_{i=1}^{s} q_i \cap p_1 \cap p_2$, it follows that
\[
\sum_{i=1}^{r} \cl(q_i) + \cl(p_1) + \cl(p_2) = 0
\]
We claim that the above relation generates the relation module of the class group. Then the claim yields the desired assertion.

Let  $\sum_{i=1}^{r} v_i \cl(q_i) + u_1 \cl(p_1) + u_2 \cl(p_2) = 0$ be an arbitrary relation in the class group $\Cl (K[\Pc])$. Then $\sum_{i=1}^{r} v_i \div(q_i) + u_1 \div(p_1)  + u_2 \div(p_2)$ is a principal divisor $\div(g)$ in  $\Div (K[\Pc])$. Since $x_c \in q_i , p_j$ for all $i$ and $j$, the divisors $\div(q_i)$ and $\div(p_j)$ are mapped to $0$ under the canonical map $\Div (K[\Pc]) \rightarrow \Div(K[\Pc]_{x_c})$. This implies that $\div(g)$ is also mapped to $0$. Hence $g$ is a unit in $K[\Pc]_{x_c}$. The only units in  $K[\Pc]_{x_c}$ are scalar multiples of powers of $x_c$, say $g= \lambda x_c ^{t}$ with $t \in \ZZ$. Therefore,
\[
\sum_{i=1}^{r} v_i \div(q_i) + u_1 \div(p_1)  + u_2 \div(p_2)= \div(g) = \div(x_c^t) = t \div(x_c).
\]
 Since  $\div(x_c) = \sum_{i=1}^{r} \div(q_i) +\div(p_1)  + \div(p_2)$, the claim holds.
\end{proof}

First, we fix some notation. As before, let $[c,d]$ be a vertical interval of $\Pc$ of maximal size, and  $e_1, \ldots, e_s$ be the elements of $[c,d]$ with the property that the maximal horizontal interval $[g_i, h_i]$ of $\Pc$ with  $e_i \in [g_i, h_i]$ contains an inside corner of $\Pc$. We furthermore let $e_0 = c$ and $e_{s+1} = d$, and for $i=0$ and $i=s+1$, we let $[g_i, h_i]$ be the maximal interval of $\Pc$ with $e_i \in [g_i, h_i]$. Now we introduce the following numbers. We set $m_j = \size ([g_j, h_j])$, for $j=0, \ldots, s$ and $m_{s+1} = 0$. Finally we set $n_{j} = \size [e_j , e_{j+1}]$, for $j=0, \ldots, s$. For the sake of uniformity, we set $q_0 = p_1$.

\begin{Theorem}
\label{canonicalclass}
Let $\cl(\omega)$ be the canonical class of $K[\Pc]$. Then
\[
\sum_{j=0}^{s} (m_j - \sum_{i=j}^{s} n_i) \cl(q_j) \quad \text {for $j=0, \ldots, s$}
 \]
is the representation of $\cl(\omega)$ with respect to the basis of $\Cl(K[\Pc])$ given in Corollary~\ref{classgroup}.
\end{Theorem}

\begin{proof}
We proceed by induction on $s$. If $s=0$, then $\Pc$ has no inside corners and desired formula follows from \cite[Theorem 8.8]{BV}. Now suppose that $s>0$. Localizing $K[\Pc]$ at $x_{h_{s+1}}$, we see that $K[\Pc]_{x_{h_{s+1}}}$ is isomorphic to the localization at $x_{h_{s+1}}$ of the polynomial ring extension $K[\Pc'][X]$ where $X= \{ x_a \: \{ a \in  [g_{s+1 , h_{s+1}}] \cup [h_s, h_{s+1}], a \neq h_s  \} \}$,  and where  $\Pc'$ is again a stack polyominoe with $n'_i = n_i$ and $m'_i = m_i - m_s$, for $i=1 , \ldots, s-1$, see Figure~\ref{local}.

\begin{figure}[hbt]
\begin{center}
\psset{unit=0.9cm}
\begin{pspicture}(4.5,-0.5)(4.5,4)
\rput(-6,0){
\pspolygon[style=fyp,fillcolor=light](4,0)(4,1)(5,1)(5,0)
\pspolygon[style=fyp,fillcolor=light](5,0)(5,1)(6,1)(6,0)
\pspolygon[style=fyp,fillcolor=light](6,0)(6,1)(7,1)(7,0)
\pspolygon[style=fyp,fillcolor=light](7,0)(7,1)(8,1)(8,0)
\pspolygon[style=fyp,fillcolor=light](8,0)(8,1)(9,1)(9,0)
\pspolygon[style=fyp,fillcolor=light](8,1)(8,2)(9,2)(9,1)
\pspolygon[style=fyp,fillcolor=light](5,1)(5,2)(6,2)(6,1)
\pspolygon[style=fyp,fillcolor=light](6,1)(6,2)(7,2)(7,1)
\pspolygon[style=fyp,fillcolor=light](7,1)(7,2)(8,2)(8,1)
\pspolygon[style=fyp,fillcolor=light](5,2)(5,3)(6,3)(6,2)
\pspolygon[style=fyp,fillcolor=light](6,2)(6,3)(7,3)(7,2)
\pspolygon[style=fyp,fillcolor=light](7,2)(7,3)(8,3)(8,2)
\pspolygon[style=fyp,fillcolor=light](6,3)(6,4)(7,4)(7,3)
\rput(7,4){$\bullet$}
\rput(7.5,4.2){$h_{s+1}$}
\rput(6.5,-0.5){$\Pc$}
}
\rput(2,0){
\pspolygon[style=fyp,fillcolor=light](5,0)(5,1)(6,1)(6,0)
\pspolygon[style=fyp,fillcolor=light](6,0)(6,1)(7,1)(7,0)
\pspolygon[style=fyp,fillcolor=light](7,0)(7,1)(8,1)(8,0)
\pspolygon[style=fyp,fillcolor=light](8,0)(8,1)(9,1)(9,0)
\pspolygon[style=fyp,fillcolor=light](8,1)(8,2)(9,2)(9,1)
\pspolygon[style=fyp,fillcolor=light](6,1)(6,2)(7,2)(7,1)
\pspolygon[style=fyp,fillcolor=light](7,1)(7,2)(8,2)(8,1)
\pspolygon[style=fyp,fillcolor=light](6,2)(6,3)(7,3)(7,2)
\pspolygon[style=fyp,fillcolor=light](7,2)(7,3)(8,3)(8,2)
\rput(7,-0.5){$\Pc'$}
}
\end{pspicture}
\end{center}
\caption{}\label{local}
\end{figure}

Since $\Cl(K[\Pc'][X]_{x_{h_{s+!}}}) = \Cl(K[\Pc'])$, we obtain a natural map $ \alpha \: \Cl(K[\Pc]) \rightarrow \Cl(K[\Pc'])$. Let $p'_1=q'_0, \ldots, q'_{s-1}, p'_2$ be the corresponding generators of $\Cl(K[\Pc'])$. Then $\alpha(\cl(q_i)) = \cl(q'_i)$ for $i=0, \ldots, s-1$, and $\alpha(\cl(q_s)) = \cl(p'_2) = - \sum_{i=0}^{s-1} \cl (q'_i)$.

Let $\cl(\omega) = \sum_{i=0}^{s} \mu_i \cl(q_i)$. Since the canonical $\cl(\omega)$ of $K[\Pc]$ is mapped to the canonical class $\cl(\omega')$ of $K[\Pc']$, we have

\[
\cl(\omega') = \sum_{j=0}^{s-1} \mu_i \cl(q'_i) + \mu_s \cl(q'_s) = \sum_{i=0}^{s-1} (\mu_i - \mu_s) \cl(q'_i).
\]

Applying the induction hypothesis we have
\begin{eqnarray}
\label{eqn}
\mu_i - \mu_s = m'_i - \sum_{j=i}^{s-1} n'_i = m_i - m_s - \sum_{i=1}^{s-1} n_i
\end{eqnarray}

Localizing $K[\Pc]$ at the variables corresponding to the outside corners of $\Pc$ different from $g_0, h_0, g_{s+1}$ and $h_{s+1}$ and using again \cite[Theorem 8.8]{BV}, we see that $\mu_s = m_s - n_s$. Hence the desired formula follows from (\ref{eqn}).
\end{proof}

As an immediate consequence of Theorem~\ref{canonicalclass}, we have the following
\begin{Corollary}
The $K$-algebra $K[P]$ is Gorenstein if and only if $m_i = \sum_{j=i}^{s} n_j$, for $i=0, \ldots, s$.
\end{Corollary}

Figure~\ref{gorenstein} shows example of a Gorenstein stack polyomino and a non-Gorenstein stack polyomino.

\begin{figure}[hbt]
\begin{center}
\psset{unit=0.9cm}
\begin{pspicture}(4.5,-0.5)(4.5,4)
\rput(-6.5,0){
\pspolygon[style=fyp,fillcolor=light](5,0)(5,1)(6,1)(6,0)
\pspolygon[style=fyp,fillcolor=light](6,0)(6,1)(7,1)(7,0)
\pspolygon[style=fyp,fillcolor=light](7,0)(7,1)(8,1)(8,0)
\pspolygon[style=fyp,fillcolor=light](8,0)(8,1)(9,1)(9,0)
\pspolygon[style=fyp,fillcolor=light](5,1)(5,2)(6,2)(6,1)
\pspolygon[style=fyp,fillcolor=light](6,1)(6,2)(7,2)(7,1)
\pspolygon[style=fyp,fillcolor=light](7,1)(7,2)(8,2)(8,1)
\pspolygon[style=fyp,fillcolor=light](5,2)(5,3)(6,3)(6,2)
\pspolygon[style=fyp,fillcolor=light](6,2)(6,3)(7,3)(7,2)
\pspolygon[style=fyp,fillcolor=light](6,3)(6,4)(7,4)(7,3)
\rput(7,-0.5){Gorenstein}
}
\rput(2,0){
\pspolygon[style=fyp,fillcolor=light](4,0)(4,1)(5,1)(5,0)
\pspolygon[style=fyp,fillcolor=light](5,0)(5,1)(6,1)(6,0)
\pspolygon[style=fyp,fillcolor=light](6,0)(6,1)(7,1)(7,0)
\pspolygon[style=fyp,fillcolor=light](7,0)(7,1)(8,1)(8,0)
\pspolygon[style=fyp,fillcolor=light](8,0)(8,1)(9,1)(9,0)
\pspolygon[style=fyp,fillcolor=light](5,1)(5,2)(6,2)(6,1)
\pspolygon[style=fyp,fillcolor=light](6,1)(6,2)(7,2)(7,1)
\pspolygon[style=fyp,fillcolor=light](7,1)(7,2)(8,2)(8,1)
\pspolygon[style=fyp,fillcolor=light](5,2)(5,3)(6,3)(6,2)
\pspolygon[style=fyp,fillcolor=light](6,2)(6,3)(7,3)(7,2)
\pspolygon[style=fyp,fillcolor=light](6,3)(6,4)(7,4)(7,3)
\rput(6.5,-0.5){Not Gorenstein}
}
\end{pspicture}
\end{center}
\caption{}\label{gorenstein}
\end{figure}

{}

\end{document}